\documentclass[reqno]{amsart}
\usepackage{url}
\newtheorem{theorem}{Theorem}[section]
\newtheorem{lemma}[theorem]{Lemma}
\newtheorem{corollary}[theorem]{Corollary}
\usepackage[english]{babel}
\usepackage[utf8]{inputenc}
\usepackage{amssymb}
\usepackage{graphicx}
\usepackage[colorinlistoftodos]{todonotes}
\usepackage{amsmath}
\usepackage{amsthm}
\usepackage{hyperref}

\theoremstyle{definition}
\newtheorem{definition}[theorem]{Definition}

\theoremstyle{remark}

\numberwithin{equation}{section}
\usepackage{algorithmic}
\usepackage[top=3cm,bottom=2cm,right=2cm,left=2cm]{geometry}

\begin{document}

\title{A Method of Verifying Partition Congruences by Symbolic Computation}


\author{Cristian-Silviu Radu}
\address{}
\curraddr{}
\email{}
\thanks{}

\author{Nicolas Allen Smoot}
\address{}
\curraddr{}
\email{}
\thanks{}


\date{}

\dedicatory{}

\maketitle

\begin{abstract}
Conjectures involving infinite families of restricted partition congruences can be difficult to verify for a number of individual cases, even with a computer.  We demonstrate how the machinery of Radu's algorithm may be modified and employed to efficiently check a very large number of cases of such conjectures.  This allows substantial evidence to be collected for a given conjecture, before a complete proof is attempted.
\end{abstract}

\begin{align*}
&\textit{``...for certain things first became clear to me by a}\\ 
&\textit{ mechanical method, although they had to be}\\
&\textit{ demonstrated by geometry afterwards because their}\\
&\textit{ investigation by the said method did not furnish an}\\
&\textit{ actual demonstration. But it is of course easier, when}\\
&\textit{ we have previously acquired... some knowledge of the}\\
&\textit{ questions, to supply the proof than it is to find it}\\
&\textit{ without any previous knowledge."}\\
&\ \ \ \ \ \ \ \ \ \ \ \ \ \ \ \ \ \ \ \ \ \ \ \ \ \ \ \ \ \ \ \ \ \ \ \ \text{---Archimedes, }\textit{The Method}
\end{align*}

\section{Introduction}

In a recent paper, one of the authors demonstrated \cite{Smoot} the proof of a conjecture of Choi, Kim, and Lovejoy:

\begin{theorem}
Let $R_l(n)$ be the number of partitions of $m$ containing a subpartition of length $l$ in which the parts are nonrepeating, nonconsecutive, and larger than all remaining parts of the partition.  If $A_1(n) = \sum_{l\ge 1}l\cdot R_l(n)$, and\\ $24n\equiv 1\pmod{5^{2\alpha}}$, then
\begin{align*}
A_1(n)\equiv 0\pmod{5^{\alpha}}.
\end{align*}
\end{theorem}  This theorem can be shown to be equivalent to the following:

\begin{theorem}
Let
\begin{align}
\displaystyle\sum_{n=0}^{\infty}a(n)q^n = \prod_{m=1}^{\infty}\frac{(1-q^{2m})^5}{(1-q^m)^3(1-q^{4m})^2}.\label{gena}
\end{align}

If $24n\equiv 1\pmod{5^{2\alpha}}$, then
\begin{align}
a(n)\equiv 0\pmod{5^{\alpha}}.
\end{align}
\end{theorem}

The methods used in proving this conjecture are based largely on the techniques developed by Paule and Radu in \cite{Paule}, which are themselves generalizations of the original techniques developed by Watson \cite{Watson} and Atkin \cite{Atkin}.

However, while these methods are powerful, and often yield elegant proofs of conjectures involving infinite families of partitions (for example, \cite{Watson}, \cite{Atkin}, \cite{Paule}, and \cite{Smoot}), they give comparatively little understanding of how these conjectures came to be inferred in the first place.  

It was noted by Choi, Kim, and Lovejoy \cite[Section 6]{CKL} that there was a close resemblance between the generating function for $a(n)$ and that of $c\phi_2(n)$, the counting function for generalized 2-colored Frobenius partitions of $n$.  Given that Paule and Radu had recently proven the Andrews--Sellers conjecture \cite{Paule}, which predicted the existence of a family of congruences for $c\phi_2(n)$, Choi, Kim, and Lovejoy suggested that a similar infinite family of congruences must exist for $a(n)$, and by extension $A_1(n)$.  They proved that

\begin{align*}
a(25n+24)\equiv 0\pmod{5},
\end{align*} and suggested additional congruences for $a(n)$ to higher powers of 5.

At first sight, the matter of specifying a family of congruences might seem easy enough.  Certainly, one could directly compute a list of the numerical values of $a(mn+j)$ for a fixed $m,j\in\mathbb{Z}_{\ge 0}$, as $n$ varies over a large number of nonnegative integers.  We could program a computer to check the greatest common divisor of this list.

Yet more interesting, one of the authors has developed algorithms \cite{Radu} that can take series of the form\\ $\sum_{n=0}^{\infty} a(mn+j)q^n$ and expand them into a finite, linear combination of eta quotients.  By examining the coefficients of each term in such a finite combination, and knowing that each eta quotient expands into an integer power series, we can often determine whether $a(mn+j)$ is divisible by a given power of a prime (in our case, 5) for \textit{ all } $n\in\mathbb{Z}_{\ge 0}$.

However, for $24n\equiv 1\pmod{5^{2\alpha}}$, one can quickly show that

\begin{align*}
n = 5^{2\alpha}\delta + \lambda_{2\alpha},
\end{align*} with $\delta\in\mathbb{Z}_{\ge 0}$ and

\begin{align*}
\lambda_{2\alpha} = \frac{23\cdot 5^{2\alpha}+1}{24}.
\end{align*}  This immediately implies that modest increases in $\alpha$ will drive even the smallest values of $n$ to increase exponentially.  Given that \cite[Chapter 6]{Andrews} $a(n)$ already increases subexponentially with $n$, it is very clear that even the most powerful computers will not be able to check the resulting expressions for 
\begin{align*}
\sum_{n=0}^{\infty} a(5^{2\alpha}n+\lambda_{2\alpha})q^n
\end{align*} beyond the very smallest values of $\alpha$.

This of course serves little concern for a conjecture already proven.  However, the methods developed by Atkin, Paule, Radu, and others to actually prove a conjecture of this sort are generally difficult.  One would of course prefer to attempt a proof only for a conjecture that already has substantial evidence in its favor.  This means that we will need to find a more efficient way to verify a family of congruences for many specific values of $\alpha$.

In this report we give one such approach.  We will use Theorem 2 above as our principal example, but we also demonstrate that these techniques may be adapted with relatively little difficulty to many similar conjectures in which an arithmetic sequence $a(n)$ has a generating function that is (up to an exponential factor) an eta quotient.

We begin in Section 2 by discussing the necessary preliminaries.  We give some information about identifying and manipulating the cusps of $\mathrm{X}_0(N)$, the modular curve corresponding to the congruence subgroup $\Gamma_0(N)$.  We then give a quick review of the theory of modular functions, Dedekind's $\eta$ function, and the $U_{\ell}$ operator.

In Section 3 we discuss the generating function (\ref{gena}) and outline the key algorithmic steps to check Theorem 2 for a large number of $\alpha$.  We make use of an important theorem whose proof can be found in \cite{Radu}, which allows us to construct a useful algebra basis for the space of eta quotients over $\Gamma_0(N)$.  From here we show how the basis can be suitably modified to interact more carefully with the $U_5$ operator.

We then discuss how to apply our method in other circumstances.  In Section 4 we briefly outline how our method can be used to efficiently check multiple cases of the Andrews--Sellers conjecture (which was proved by Paule and Radu \cite{Paule}).  We give a generalized form of our method in Section 5.  Finally, in Section 6, we explain why our approach, so useful in verifying a substantial number of cases of a conjecture, is not capable of providing a complete proof.

\section{Preliminaries}

Henceforth, we will denote $\mathbb{H}$ as the upper half complex plane, with $\tau\in\mathbb{H}$, and $q = e^{2\pi i\tau}$.  Furthermore, we will denote

\begin{align*}
(q^a;q^b)_{\infty} := \prod_{m=0}^{\infty}\left(1 - q^{a+bm}\right).
\end{align*}

\subsection{$\Gamma_0(N)$}

Let $N\in\mathbb{Z}_{>0}$.  We will denote 

\begin{align*}
\mathrm{SL}(2,\mathbb{Z}) := \Bigg\{ \begin{pmatrix}
  a & b \\
  c & d 
 \end{pmatrix}:\ a,b,c,d\in\mathbb{Z},\ ad-bc=1 \Bigg\}.
\end{align*}  Furthermore, we let

\begin{align*}
\Gamma_0(N) := \Bigg\{ \begin{pmatrix}
  a & b \\
  c & d 
 \end{pmatrix}\in \mathrm{SL}(2,\mathbb{Z}): N|c \Bigg\},
\end{align*} and

\begin{align*}
\mathrm{SL}(2,\mathbb{Z})_{\infty} := \Bigg\{ \begin{pmatrix}
  1 & b \\
  0 & 1 
 \end{pmatrix} : b\in \mathbb{Z} \Bigg\}.
\end{align*}

\begin{definition}
Let $a/c\in\mathbb{Q}\cup\{\infty\}$.  The cusp over $\Gamma_0(N)$ represented $a/c$ is the coset

\begin{align*}
\Gamma_0(N)\cdot\frac{a}{c} := \left\{ \frac{a_0\cdot\frac{a}{c}+b_0}{c_0\cdot\frac{a}{c}+d_0}: \begin{pmatrix}
  a_0 & b_0 \\
  c_0 & d_0 
 \end{pmatrix}\in\Gamma_0(N) \right\}.
\end{align*}  If $a_1/c_1\in\Gamma_0(N)\cdot\frac{a}{c}$, then $a_1/c_1$ represents the same cusp as $a/c$.
\end{definition}  Given any $N\in\mathbb{Z}_{>0}$, the cusps over $\Gamma_0(N)$ form a set of equivalence classes of $\mathbb{Q}$.  Indeed, \cite[Proposition 3.8.5]{Diamond}, the number of distinct cusps over $\Gamma_0(N)$ matches the number of double cosets of 

\begin{align*}
\Gamma_0(N)\backslash \mathrm{SL}(2,\mathbb{Z})/\mathrm{SL}(2,\mathbb{Z})_{\infty}.
\end{align*}  Because the index of $\Gamma_0(N)$ over $\mathrm{SL}(2,\mathbb{Z})$ is finite \cite[Section 1.2]{Diamond}, the number of cusps over $\Gamma_0(N)$ must necessarily be finite.

The following theorem \cite[Proposition 3.8.3]{Diamond} gives a condition for determining whether two elements of $\mathbb{Q}\cup\{\infty\}$ represent the same cusp.

\begin{theorem}
Let $a/c,\ a_1/c_1\in\mathbb{Q}\cup\{\infty\}$ with $\mathrm{gcd}(a,c) = \mathrm{gcd}(a_1,c_1) = 1$.  Then $a_1/c_1$ represents the same cusp over $\Gamma_0(N)$ as $a/c$ if and only if there exist integers $m,n\in\mathbb{Z}$ such that

\begin{align*}
m a_1 &\equiv a+n c\pmod{N},\\
c_1 &\equiv m c\pmod{N},
\end{align*} with $\mathrm{gcd}(m,N)=1$.
\end{theorem}  \noindent A complete treatment of the geometrical interpretation of the cusps over $\Gamma_0(N)$ and the associated modular curve $\mathrm{X}_0(N)$ can be found in \cite[Chapters 2, 3]{Diamond} .

\subsection{Modularity}

\begin{definition}
Let $q=e^{2\pi i\tau}$, with $\tau\in\mathbb{H}$, and suppose that $f:\mathbb{H}\rightarrow\mathbb{C}$ is a holomorphic function for all $\tau\in\mathbb{H}$.  In this case, $f$ is a weakly holomorphic modular form over $\Gamma_0(N)$ with weight $k\in\mathbb{Z}$ if the following conditions apply:

\begin{enumerate}
\item  For any $\begin{pmatrix}
  a & b \\
  c & d 
 \end{pmatrix}\in\Gamma_0(N)$, we have \begin{align*}
\left(c\tau+d\right)^{-k}\cdot f\left( \frac{a\tau+b}{c\tau+d} \right) = f(\tau),
\end{align*}
\item For any $\gamma=\begin{pmatrix}
  a & b \\
  c & d 
 \end{pmatrix}\in \mathrm{SL}(2,\mathbb{Z})$, we have \begin{align*}
\left(c\tau+d\right)^{-k}\cdot f\left( \frac{a\tau+b}{c\tau+d} \right) = \sum_{n=n_{\gamma}(f)}^{\infty}\alpha_{\gamma}(n)q^{n\cdot\mathrm{gcd}(c^2,N)/ N},
\end{align*} with $n_{\gamma}(f)\in\mathbb{Z}$, and $\alpha_{\gamma}(n)\in\mathbb{C}$ for all $n\ge n_{\gamma}(f)$, with $\alpha_{\gamma}(n_{\gamma}(f))\neq 0$.
\end{enumerate}  Here, we define $\mathrm{ord}_{a/c}^{(N)}(f) := n_{\gamma}(f)$ as the order of $f$ at the cusp represented by $a/c$, over $\Gamma_0(N)$.  If $\mathrm{ord}_{a/c}^{(N)}(f)<0$, then $f$ is said to have a pole at $a/c$, with principal part

\begin{align*}
\sum_{n=n_{\gamma}(f)}^{-1} \alpha_{\gamma}(n) q^{n\cdot\mathrm{gcd}(c^2,N)/ N}.
\end{align*}  If $\mathrm{ord}_{a/c}^{(N)}(f) > 0$, then $f$ is said to have a zero at $a/c$.

\end{definition}

If we strengthen our first condition, i.e., if $k=0$, then we call $f$ a modular function.  If we strengthen the second condition, i.e., if we insist that $\mathrm{ord}_{a/c}^{(N)}(f)\ge 0$ for every cusp of $\Gamma_0(N)$, then we call $f$ a modular form of weight~$k$.

An extremely important property of weakly holomorphic modular forms is that these conditions cannot both be strengthened without reducing the relevant functions to a constant.

\begin{theorem}
Let $N\in\mathbb{Z}_{>0}$.  If $f$ is a modular function with nonnegative order at every cusp of $\Gamma_0(N)$, then $f$ must be a constant.
\end{theorem}  See \cite[Chapter 2, Theorem 7]{Knopp} for a proof.

This has been called ``the fundamental theorem of the subject [of modular functions]" \cite[Chapter 1, Section 3]{Lehner}.  Its utility becomes clear upon comparing any two modular functions.  If $f,g$ are both modular functions over $\Gamma_0(N)$, and their principal parts at each of their poles match, then $f-g$ must be a modular function with no poles at any cusp.  This forces $f-g$ to be a constant.  If their constants also match, then $f$ and $g$ must be equal, since $f-g=0$.

The question of equality between modular functions can therefore be reduced to the question of comparing their finite principal parts and constants---which of course immediately reduces to the question of comparing polynomials.

Hereafter, we will denote $\mathcal{M}(N)$ as the set of all modular functions over $\Gamma_0(N)$.  For any field $\mathbb{K}$, we define $\mathcal{M}(N)_{\mathbb{K}}$ as the set of modular functions $f\in\mathcal{M}(N)$ such that $\alpha_{\mathrm{I}}(n)\in\mathbb{K}$ for all $n\ge n_{\mathrm{I}}(f)$. Finally, define $\mathcal{M}^{\infty}(N)$ as the set of all modular functions over $\Gamma_0(N)$ in which $n_{\gamma}\ge 0$ for every $\gamma\in~\mathrm{SL}(2,\mathbb{Z})\backslash\Gamma_0(N)$.  That is, the functions of $\mathcal{M}^{\infty}(N)$ only have a single pole (not counting multiplicity), at the cusp represented by $\infty$.

Finally, given a ring $\mathcal{R}$ and a set of functions $\mathcal{S}$ over $\mathbb{C}$, we let $\left< \mathcal{S} \right>_{\mathcal{R}}$ denote all elements of the form $r_1s_1 + r_2s_2 + ... + r_ks_k$, for any $r_1, r_2,..., r_k\in\mathcal{R}$ and any $s_1,s_2,...,s_k\in\mathcal{S}$.

\subsection{Dedekind's $\eta$ Function}

Of particular importance to us is Dedekind's eta function, which we define here:

\begin{definition}
\begin{align*}
\eta(\tau) := e^{\pi i\tau/12}\prod_{n=1}^{\infty} (1 - e^{2\pi i n\tau}) = q^{1/24} (q;q)_{\infty}.
\end{align*}
\end{definition}  The function $\eta$ is not strictly modular by our definitions.  However, it does satisfy a slightly weaker symmetric condition (that is, $\eta$ is a modular form of fractional weight over $\mathrm{SL}(2,\mathbb{Z})$ with a nontrivial multiplier system \cite[Chapter 3, Theorem 10]{Knopp}).  This, combined with the many combinatorial interpretations of $\eta$, make it useful in constructing and representing many important modular functions.

For this purpose, we introduce a theorem due to Newman \cite[Theorem 1]{Newman2}:

\begin{theorem}
Let $f = \prod_{\delta | N} \eta(\delta\tau)^{r_{\delta}}$, with $r = (r_{\delta})_{\delta | N}$ an integer-valued vector, for some $N\in\mathbb{Z}_{>0}$.  Then $f$ is a modular function over $\Gamma_0(N)$ if and only if the following apply:

\begin{align}
&\sum_{\delta | N} r_{\delta} = 0,\label{newman1a}\\
&\sum_{\delta | N} \delta r_{\delta} \equiv 0\pmod{24},\label{newman1b}\\
&\sum_{\delta | N} \frac{N}{\delta}r_{\delta} \equiv 0\pmod{24},\label{newman1c}\\
&\prod_{\delta | N} \delta^{|r_{\delta}|} = k_0^2,\label{newman1d}
\end{align} for some $k_0\in\mathbb{Z}$.

\end{theorem}  An example of such a modular function is

\begin{align*}
\left(\frac{\eta(5\tau)}{\eta(\tau)}\right)^6 = \eta(\tau)^{-6}\eta(5\tau)^6\in\mathcal{M}(5),
\end{align*} since $r=(-6,6)$ satisfies the four conditions above.

\begin{definition}
An eta quotient over $\Gamma_0(N)$ is a function of the form

\begin{align*}
f = \prod_{\delta | N} \eta(\delta\tau)^{r_{\delta}},
\end{align*} with $r = (r_{\delta})_{\delta | N}$ an integer-valued vector.

Define $\mathcal{E}(N)$ to be the set of eta quotients which are modular functions over $\Gamma_0(N)$, and $\mathcal{E}^{\infty}(N) := \mathcal{M}^{\infty}(N)\cap\mathcal{E}(N)$.
\end{definition}

Given an eta quotient $f$, its expansion at $\infty$ has integer coefficients, as does its inverse $1/f$.  Moreover, we have a precise formula for the order of $f$ at any given cusp, as given in \cite[Theorem 23]{Radu}, generally attributed to Ligozat:

\begin{theorem}
If $f = \prod_{\delta | N} \eta(\delta\tau)^{r_{\delta}}\in\mathcal{E}(N)$, then the order of $f$ at the cusp represented by $a/c$ is given by the following:

\begin{align}
\mathrm{ord}_{a/c}^{(N)}(f) = \frac{N}{24\gcd{(c^2,N)}}\sum_{\delta | N} r_{\delta}\frac{\gcd{(c,\delta)}^2}{\delta}.\label{ligozat11}
\end{align}

\end{theorem}

\subsection{$U_{\ell}$ Operator}

We recall the classic $U_{\ell}$-operator:

\begin{definition}

Let $\ell\in\mathbb{Z}_{>0}$ be a prime, and $f(q) = \sum_{m\ge M}a(m)q^m$.  Then define

\begin{align*}
U_{\ell}\left(f(q)\right) := \sum_{\ell\cdot m\ge M} a(\ell\cdot m)q^m.
\end{align*}

\end{definition}

This operator is often enormously useful, because it gives us a means of connecting different cases of a given congruence conjecture.

For most of our examples, we will only need the specific case $\ell = 5$, but for the remainder of the subsection we will list some key properties of $U_{\ell}$ in which $\ell$ is an arbitrary but fixed prime.  This will allow us to generalize our results in Section 5.

The properties in the following lemma are standard to the theory of partition congruences, and proofs can be found in \cite[Chapter 10]{Andrews} and \cite[Chapter 8]{Knopp}.

\begin{lemma}

Given two functions 

\begin{align*}
f(q) = \sum_{m\ge M}a(m)q^m,\ g(q) = \sum_{m\ge N}b(m)q^m,
\end{align*} any $\alpha\in\mathbb{C}$, a primitive $\ell$-th root of unity $\zeta$, and the convention that $q^{1/\ell}~=~e^{2\pi i\tau/\ell}$, we have the following:
\begin{enumerate}
\item $U_{\ell}\left(\alpha\cdot f+g\right) = \alpha\cdot U_{\ell}\left(f\right) + U_{\ell}\left(g\right)$;
\item $U_{\ell}\left(f(q^{\ell})g(q)\right) = f(q) U_{\ell}\left(g(q)\right)$;
\item $\ell\cdot U_{\ell}\left(f\right) = \sum_{r=0}^{\ell - 1} f\left( \zeta^rq^{1/\ell} \right)$.
\end{enumerate}
\end{lemma}

Finally, we give an important theorem on the stability of $U_{\ell}$ \cite[Lemma 17 (iv)]{Atkin2}.

\begin{theorem}
Let $k, N\in\mathbb{Z}_{\ge 0}$, with $\ell^2 |N$.  Then $U_{\ell}\left(f\right)\in\mathcal{M}_k(N/{\ell})$ for all $f\in\mathcal{M}_k(N)$.
\end{theorem}

\begin{proof}
From Part 3 of the lemma above, we know that

\begin{align*}
U_{\ell}\left(f\right) = \sum_{r=0}^{\ell - 1} \frac{1}{\ell}f\left( \zeta^rq^{1/\ell} \right).
\end{align*}  Here, changing variables to $\tau$, we find that

\begin{align*}
\zeta^rq^{1/\ell} = \exp\left(\frac{2\pi i(r + \tau)}{\ell}\right).
\end{align*}  Because $f$ is holomorphic for $\tau\in\mathbb{H}$, $\frac{2\pi i(r + \tau)}{\ell}\in\mathbb{H}$, therefore $U_{\ell}(f)$ must be holomorphic as well.

Letting $\tilde{f}(\tau) = f(e^{2\pi i\tau})$, and letting $\gamma=\begin{pmatrix}
  a & b \\
  c & d 
 \end{pmatrix}\in\Gamma_0(N/\ell)$, we have

\begin{align*}
U_{\ell}\left(f\right)\left( \gamma\tau \right) &= \sum_{r=0}^{\ell - 1} \frac{1}{\ell}\tilde{f}\left( \frac{\gamma\tau + r}{\ell} \right)\\
&= \sum_{r=0}^{\ell - 1} \frac{1}{\ell}\tilde{f}\left( \frac{1}{\ell}\left(\frac{a\tau + b}{c\tau + d}\right) + \frac{r}{\ell} \right)\\
&= \sum_{r=0}^{\ell - 1} \frac{1}{\ell}\tilde{f}\left( \frac{(a+rc)\tau + (b+rd)}{\ell\cdot c\tau + \ell\cdot d}\right)\\
&= \sum_{r=0}^{\ell - 1} \frac{1}{\ell}\tilde{f}\left( \gamma' \tau \right),
\end{align*} with

\begin{align*}
\gamma' = \begin{pmatrix}
  a+rc & b+rd \\
  \ell\cdot c & \ell\cdot d 
 \end{pmatrix}.
\end{align*}  Because $\gcd(a,c) = 1$, and because $\ell |c$, we have $\gcd(a+rc,\ell\cdot c)=1$.  Therefore, there exist integers $x,y$, such that $(a+rc)x+\ell\cdot cy=1$.  From this, we immediately have

\begin{align*}
\gamma' = \begin{pmatrix}
  a+rc & b+rd \\
  \ell\cdot c & \ell\cdot d 
 \end{pmatrix} = \begin{pmatrix}
  a+rc & -y \\
  \ell\cdot c & x 
 \end{pmatrix}\begin{pmatrix}
  1 & x(b+rd)+\ell\cdot yd \\
  0 & \ell 
 \end{pmatrix}.
\end{align*}  Because $\begin{pmatrix}
  a+rc & -y \\
  \ell\cdot c & x 
 \end{pmatrix}\in\Gamma_0(N)$, we have

\begin{align*}
\tilde{f}\left( \gamma' \tau \right) &= \left(\ell\cdot c\cdot\frac{\tau + x(b+rd)+\ell\cdot yd}{\ell}+x\right)^k \tilde{f}\left( \begin{pmatrix}
  1 & x(b+rd)+\ell\cdot yd \\
  0 & \ell
 \end{pmatrix}\tau \right)\\
&=\left(c\tau+d\right)^k \tilde{f}\left( \frac{\tau + x(b+rd)+\ell\cdot yd}{\ell} \right)\\
&=\left(c\tau+d\right)^k \tilde{f}\left( \frac{\tau + rxd + (bx +\ell\cdot yd)}{\ell} \right)
\end{align*}  Now because $ad-bc=1$, we have $\gcd(c,d)=1$, implying also that $\gcd(\ell,d)=1$.  So as $r$ runs through the residues modulo $\ell$, $rxd$ similarly runs through all the residues.  Similarly, $rxd + (bx +\ell\cdot yd)$ must run through all the residues too.  In other words,

\begin{align*}
U_{\ell}\left(f\right)\left( \gamma\tau \right) &= \sum_{r=0}^{\ell - 1} \frac{1}{\ell}\tilde{f}\left( \frac{\gamma\tau + r}{\ell} \right)\\
&= \sum_{r=0}^{\ell - 1} \frac{1}{\ell}(c\tau+d)^k\tilde{f}\left( \frac{\tau + rxd + (bx +\ell\cdot yd)}{\ell} \right)\\
&=(c\tau+d)^k\cdot\frac{1}{\ell} \sum_{r=0}^{\ell} \tilde{f}\left( \frac{\tau + r}{\ell} \right)\\
&=(c\tau+d)^k\cdot U_{\ell}\left(f\right)(\tau).
\end{align*}

\end{proof}

\section{Rogers--Ramanujan Subpartitions}

With the necessary preliminaries established, we will begin with the example from which we first developed our method.

Define the sequence $a(n)$ by the following:

\begin{align*}
\mathrm{C}(q) := \sum_{n=0}^{\infty}a(n)q^n = \frac{(q^2;q^2)^5_{\infty}}{(q;q)^3_{\infty}(q^4;q^4)^2_{\infty}}.
\end{align*}  A possible infinite family of congruences modulo 5 was suggested by Choi, Kim, and Lovejoy, although they did not specify the exact family of congruences.  However, it is easy for us to investigate suspicious cases.

\subsection{A Suspicious Case}

The most obvious case to check is the condition that $24n\equiv 1\pmod{5^k}$, which is the congruence condition defining Ramanujan's classic cases.  It can be readily checked that no congruence is obtained in the case that $k=1$.

On the other hand, Choi, Kim, and Lovejoy proved that $a(25n+24)\equiv 0\pmod{5}$, which corresponds to $k=2$.  One may, with some mild computational difficulty, verify that $a(125n+99)$ is always divisible by 5, but not by 25.  One might be inclined to suggest that the interesting condition to examine is 

\begin{align*}
24n\equiv 1\pmod{5^{2k}}.
\end{align*}  However, we have only checked a single case.  Examining even the very next case, the progression $625n+599$, will be much more difficult, especially as $a(n)$ already grows subexponentially.

But supposing that this is indeed the case, let us determine how we might check it and successive cases.  Our opening steps are not unlike the initial steps to a true proof.  We can very quickly define a sequence of functions $\mathcal{L} = (L_{\alpha})_{\alpha\ge 0}$ in which

\begin{align*}
L_0 &:= 1,\\
L_{\alpha} &:= \Phi_{\alpha}\cdot \sum_{24n\equiv 1\bmod{5^{\alpha}}} a(n)q^{\left\lfloor n/5^{\alpha} \right\rfloor},
\end{align*}

\begin{align*}
A := q\cdot\frac{\mathrm{C}(q)}{\mathrm{C}(q^{25})},
\end{align*}

\begin{align*}
U^{(0)}\left(f\right) &:= U_5 \left(A\cdot f\right),\\
U^{(1)}\left(f\right) &:= U_5\left(f\right),\\
U^{(\alpha)}\left(f\right) &:= U^{(\alpha\bmod{2})}\left(f\right),
\end{align*} and

\begin{align*}
L_{\alpha+1} &= U^{(\alpha)} \left(L_{\alpha}\right).
\end{align*}  This very quickly yields the following functions for $\Phi_{\alpha}$:

\begin{align*}
\Phi_{2\alpha - 1} &= \frac{q}{\mathrm{C}(q^5)}, \text{ and } \Phi_{2\alpha} = \frac{q}{\mathrm{C}(q)}.
\end{align*}

We want to know whether $\mathcal{L_{\alpha}}$ converges 5-adically to 0.  As an example, let us select $L_1$:

\begin{align*}
L_1 &= \frac{(q^5;q^5)^3_{\infty}(q^{20};q^{20})^2_{\infty}}{(q^{10};q^{10})^5_{\infty}}\sum_{m=0}^{\infty} a(5m+4)q^{m+1}\\
& = \frac{\eta(5\tau)^3\eta(20\tau)^2}{\eta(10\tau)^5}\cdot q^{1-5/24}\sum_{m=0}^{\infty} a(5m+4)q^{m}.
\end{align*}

In keeping with the notation of \cite[Section 3]{Radu}, we have $M=4, \hat{r}=(-3,5,-2),$ $m=25, t=24$.  The smallest possible value of $N$ to satisfy the $\Delta^{\ast}$ criteria is $N=20$.

Now, the vector $\hat{s} = (0,0,0,3,-5,2)$ satisfies the conditions of \cite[Theorem 45]{Radu}, and $P_{5,\hat{r}}(4) = \{4\}$ \cite[Definition 42]{Radu}.  Finally, we have

\begin{align*}
\alpha = 1-\frac{5}{24} = \frac{19}{24} = \frac{4}{5} + \frac{1}{120} (1(-3) + 2(5) + 4(-2)).
\end{align*}

We have therefore shown that $L_1\in\mathcal{M}(20)$.

Because $U_5^{(\alpha)}\left(f\right)\in\mathcal{M}(20)$ for all $f\in\mathcal{M}(20)$, we have that $\mathcal{L}$ forms a sequence of functions in $\mathcal{M}(20)$.

However, while $L_1\in\mathcal{M}(20)$, it is not necessarily in $\mathcal{M}^{\infty}(20)$.  We need some $\omega\in\mathcal{E}^{\infty}(20)$ that will overcome any other poles that $L_1$ has, i.e.,

\begin{align*}
\omega\cdot L_1\in\mathcal{M}^{\infty}(20)_{\mathbb{Q}}.
\end{align*}  We may take advantage of the fact that 

\begin{align}
\mathcal{M}^{\infty}(20)_{\mathbb{Q}} = \left< \mathcal{E}^{\infty}(20) \right>_{\mathbb{Q}}.
\end{align}  A proof of this will be given in Section 3.2, but for now we take it for granted.  In that case, $\omega\cdot L_1\in\left< \mathcal{E}^{\infty}(20) \right>_{\mathbb{Q}}$.

In order to give the exact expression of $\omega\cdot L_1\in\left< \mathcal{E}^{\infty}(20) \right>_{\mathbb{Q}}$, we take advantage of an algorithm given in \cite{Radu} to produce the following algebra basis for $\left< \mathcal{E}^{\infty}(20) \right>_{\mathbb{Q}}$

\begin{theorem}
Given $N\in\mathbb{Z}_{>0}$, there exist functions $t, g_1, g_2, ..., g_{v}\in \left<\mathcal{E}^{\infty}(N)\right>_{\mathbb{Q}}$ such that for all $i,j$ with $0\le i<j \le v-1$ (with $g_0=1$),
\begin{itemize}
\item $|\mathrm{ord}^{(N)}_{\infty}(t)| = v+1,$
\item $|\mathrm{ord}^{(N)}_{\infty}(g_i)| < |\mathrm{ord}^{(N)}_{\infty}(g_j)|$,
\item $|\mathrm{ord}^{(N)}_{\infty}(g_i)| \not\equiv |\mathrm{ord}^{(N)}_{\infty}(g_j)|\pmod{v+1}$,
\item $|\mathrm{ord}^{(N)}_{\infty}(g_i)| \not\equiv 0\pmod{v+1}$\text{ except when $i=0$},
\item $\left< \mathcal{E}^{\infty}(N) \right>_{\mathbb{Q}} = \left< 1, g_1, ..., g_{v} \right>_{\mathbb{Q}[t]}.$
\end{itemize}
\end{theorem}  To understand the importance of this theorem, let us suppose that\\ $f\in\mathcal{M}^{\infty}(N)_{\mathbb{Q}}$.  We want to determine its membership in $\left< \mathcal{E}^{\infty}(N) \right>_{\mathbb{Q}} = \left< 1, g_1, ..., g_{v} \right>_{\mathbb{Q}[t]}$, in which we assume that the functions $g_j, t$ satisfy the conditions of Theorem 8.  We will describe the membership check algorithm (MW) from \cite{Radu}.

To begin, we set $k=0$, $f_0=f$, and define $m_k := -\mathrm{ord}_{\infty}^{(N)}(f_k)\in\mathbb{Z}_{\ge 0}$.

\begin{enumerate}
\item If $m_k=0$, then go to Step 9.  Otherwise, proceed to Step 2.
\item Expand the principal part and constant of $f_k$, which we represent here as
\begin{align*}
\frac{b(-m_k)}{q^{m_k}} + \frac{b(-m_k+1)}{q^{m_k-1}} + ... + \frac{b(-1)}{q} + b(0),
\end{align*} with $b(n)\in\mathbb{Q}$ for $-m_0\le n\le 0$, and $b(-m_k)\neq 0$.
\item Examine $m_k\pmod{v+1}$.  Notice that the orders of our algebra basis functions give us a complete set of residues modulo $v+1$.  There must be one and only one function $g_{i_k}$ (which may also be $t$) in our basis with a matching residue class.
\item If $|\mathrm{ord}_{\infty}^{(N)}(g_{i_k})| \ge m_k$, then for any $n\in\mathbb{Z}_{\ge 0}$ and any $\alpha\in\mathbb{Z}$,
\begin{align*}
|\mathrm{ord}_{\infty}^{(N)}(f_k - \alpha\cdot g_{i_k}t^n)| \ge m_k.
\end{align*}  Moreover, because $-m_k$ is not equivalent to the order of any other basis element modulo $v+1$, the order of $f_k$ cannot be reduced with respect to $\left< 1, g_1, ..., g_{v} \right>_{\mathbb{Q}[t]}$, and we have disproved membership of $f$, and we end the algorithm.
\item If $0 < |\mathrm{ord}_{\infty}^{(N)}(g_{i_k})| = n_k < m_k$, then we may write
\begin{align*}
m_{k+1} = \Big |\mathrm{ord}_{\infty}^{(N)}\left(f_k - \frac{b(-m_k)}{\mathrm{LC}(g_{i_k}\cdot t^{(n_k - m_k)/(v+1)})} g_{i_k}\cdot t^{(n_k - m_k)/(v+1)}\right) \Big | < m_k,
\end{align*} in which $\mathrm{LC}(h)$ is defined as the leading coefficient of $h$ (for example, 
\\$\mathrm{LC}(f_k) = b(-m_k)$).
\item Now let
\begin{align*}
f_{k+1} = f_k - \frac{b(-m_k)}{\mathrm{LC}(g_{i_k}\cdot t^{(n_k - m_k)/(v+1)})} g_{i_1}\cdot t^{(n_k - m_k)/(v+1)}\in\mathcal{M}^{\infty}(N)_{\mathbb{Q}}.
\end{align*}
\item Set $k=k+1$.
\item If $m_k=0$, then proceed to Step 9.  Otherwise, return to Step 3.
\item If $m_k = 0$, then we have reduced the entire principal part of $f$ to combinations of the principal parts of $t, g_1, ..., g_v$.  Since for each $j$ such that $0\le j\le k$, $f_j$ has rational coefficients, the constant term of $f_k$ must be a rational multiple of 1.  We have therefore demonstrated membership, and we end the algorithm.
\end{enumerate}

Through the MC algorithm, we construct a strictly decreasing sequence of numbers $\{m_0, m_1, m_2, ...\}$ in $\mathbb{Z}_{\ge 0}$.  Such a sequence cannot continue indefinitely, so that we must either disprove membership, or reach $m_M = 0$ for some $M\in\mathbb{Z}_{\ge 0}$, in a finite number of steps.

Theorem 8 therefore gives us a computational means of determining membership in $\left< \mathcal{E}^{\infty}(N) \right>_{\mathbb{Q}}$.

Let us define $\left< \mathcal{E}^{\infty}(20) \right>_{\mathbb{Q}} = \left< 1, G_1, ..., G_v \right>_{\mathbb{Q}[T]},$ in which $T, G_1, ..., G_v$ satisfy the conditions of Theorem 8.

We note that $\omega,T,T^{-1},G_i\in\mathcal{E}(20)$.  Let us suppose for the moment that each of these functions has integer coefficients in its $q$-expansion.  We must therefore have polynomials $p_0, p_1, ..., p_v\in\mathbb{Z}[x]$ such that

\begin{align}
\omega\cdot L_1 &= p_0(T) + p_1(T)G_1 + ... + p_v(T)G_v,\\
L_1 &= \frac{p_0(T)}{\omega} + \frac{p_1(T)}{\omega}G_1 + ... + \frac{p_v(T)}{\omega}G_v.\label{L1E1}
\end{align}

If we apply $U^{(1)} = U_5$ to both sides of (\ref{L1E1}), we then have an expression for $L_2$ in terms of $U^{(1)}\left(T^j G_k/\omega\right)$.  If we were able to find appropriate expansions of these terms (e.g., expansions in terms of $T, G_k$), then we could apply $U^{\alpha}$ arbitrarily many times, and find expansions of $L_{\alpha}$, no matter the size of $\alpha$.

We can simplify matters enormously by imposing an additional condition to the necessary properties of our algebra basis.  We know that $L_1$ has poles at various cusps of $\Gamma_0(20)$.  If we were to choose $T$ to have positive order at the corresponding poles of $L_1$, then we could make the substitution $\omega = T^l$, for $l\in\mathbb{Z}_{>0}$ sufficiently large:

\begin{align*}
T^l\cdot L_1 &= p_0(T) + p_1(T)G_1 + ... + p_v(T)G_v,\\
L_1 &\in\left< 1, G_1, ..., G_v \right>_{\mathbb{Z}[T,T^{-1}]}.
\end{align*}  Now we have only to understand $U^{(i)}\left(T^j G_k\right)$, for $i\in\{0,1\}$, $j\in\mathbb{Z}$, and $k\in\{0,1,...,v\}$.  Moreover, if we are careful to arrange so that $T$ has positive order at every pole exhibited by the functions $A^{i}T^jG_k$ for all $i\in\{0,1\}$, $j\in\mathbb{Z}$, and $k\in\{0, 1,..., v\}$, then we will have

\begin{align*}
U^{(i)}\left(T^jG_k\right)\in \left< 1, G_1, ..., G_v \right>_{\mathbb{Z}[T,T^{-1}]}.
\end{align*}  That is, $\left< 1, G_1, ..., G_v \right>_{\mathbb{Z}[T,T^{-1}]}$ is closed under $U^{(\alpha)}$ for all $\alpha\in\mathbb{Z}_{\ge 0}$.

From this closure theorem, we can construct a relatively efficient algorithm for checking $L_{\alpha}$ for divisibility by powers of 5.  Supposing we want to check our conjecture that $L_{2\alpha}\equiv 0\pmod{5^{\alpha}}$, by examining $0\le \alpha\le 2B$, for some $B\in\mathbb{Z}_{>0}$.

Noting that we can define $L_0:=1$, we can begin by immediately establishing that $L_0\in\left< 1, G_1, ..., G_v \right>_{\mathbb{Z}[T,T^{-1}]}$.

From here, we compute 

\begin{align*}
L_1 = U^{(0)}\left(1\right) = \displaystyle\sum_{\substack{j\in\mathbb{Z},\\ 0\le k\le v}} b_{1,j,k}T^{j}G_k.
\end{align*}  However, as we apply $U^{(\alpha)}$ for increasing $\alpha$, we will find the coefficients $b_{1,j,k}$ become very large.  To resolve this, we reduce each coefficient to the least positive residue modulo $5^B$:

\begin{align*}
L_1^{(B)} = \displaystyle\sum_{\substack{j\in\mathbb{Z},\\ 0\le k\le v}} c_{1,j,k}T^{j}G_k,
\end{align*} with $c_{1,j,k}\equiv b_{1,j,k}\pmod{5^B}$.  We thus define the following sequence of functions:

\begin{align*}
L_0^{(B)} &:= 1,\\
L_{\alpha}^{(B)} &:= U_5^{(\alpha-1)}\left(L_{\alpha-1}^{(B)}\right) \pmod{5^B} = \displaystyle\sum_{\substack{j\in\mathbb{Z},\\ 0\le k\le v}} c_{\alpha,j,k}T^{j}G_k,
\end{align*} with $0\le c_{\alpha,j,k}<5^B$ for all $\alpha, j, k$.

We now give the steps for checking this conjecture:

\begin{enumerate}
\item Begin with $\alpha=0$, $v_0=0$, and $V = \{ v_0 \}$.
\item Expand $L_{\alpha}^{(B)}$ into $\left< 1, G_1, ..., G_v \right>_{\mathbb{Z}[T,T^{-1}]}$: $L_{\alpha}^{(B)} = \displaystyle\sum_{\substack{j\in\mathbb{Z},\\ 0\le k\le v}} c_{\alpha,j,k}T^{j}G_k$.
\item Expand $U_5^{(\alpha)}\left(L_{\alpha}^{(B)}\right) = \displaystyle\sum_{\substack{j\in\mathbb{Z},\\ 0\le k\le v}} c_{\alpha,j,k}U_5^{(\alpha)}\left(T^{j}G_k\right)$.
\item Reduce $U_5^{(\alpha)}\left(L_{\alpha}^{(B)}\right)\pmod{5^{B}}$ to get $L_{\alpha+1}^{(B)} = \displaystyle\sum_{\substack{j\in\mathbb{Z},\\ 0\le k\le v}} c_{\alpha+1,j,k}T^{j}G_k$.
\item Let $v_{\alpha+1}$ be the maximal power of 5 (up to $B$) dividing each nonzero $c_{\alpha +1,j,k}$.
\item Set $V = V\cup \{v_{\alpha+1}\}$.
\item Set $\alpha = \alpha+1$, and return to Step 2.  Continue until $\alpha=2B$.
\item If $v_{2\alpha}=\alpha$ for $0\le \alpha\le B$, then we have verified our conjecture for the first $B$ cases.  Otherwise, the conjecture fails.
\end{enumerate}  Here, the growth of our coefficients $c_{\alpha,j,k}$ is limited by the size of $5^{B}$.  This bound grows exponentially with $B$, but it is far better than the sub-double-exponential coefficient growth that we would otherwise expect.

For example, setting $B=5$ ensures that $L_{\alpha}^{(B)}$ will contain terms smaller than $5^{10}$, of the order of $10^7$.  These numbers are small enough even for a modest laptop to manage, and the conjecture can now be checked and verified for 5 distinct cases.

Now, as we will demonstrate in Section 6, we cannot use this method alone to provide a complete proof of our conjecture.  However, our method is of critical importance for two reasons.  First, it allows us to check our hastily made conjecture while investing relatively little time or computation.  If substantial evidence accumulates in its favor, we may certainly employ more difficult techniques to attempt a proof.

Secondly, our method begins with computation of a very precise algebra basis for $\left< \mathcal{E}^{\infty}(20) \right>_{\mathbb{Q}}$.  As is demonstrated in \cite[Section 4.1]{Smoot}, the functions in this basis are essential for actually completing the proof of the conjecture.  This alone establishes that the full algorithm, with its relative efficiency and economy, may as well be brought to bear before attempting a proof.

\subsection{The Basis}

Of course, the functions $T, G_1, ..., G_v$ need to be directly computed.  The functions $G_k$ may be computed using the algebra basis algorithm in \cite{Radu}, once the function $T$ is known.  However, $T$ must be selected with care, so that its positive-order zeros correspond with any poles possessed by $U_5\left(A^{i}T^jG_k\right)$, for all $(i,j,k)\in\{0,1\}\times\mathbb{Z}\times\{0, 1,..., v\}.$

To begin with, we assume that $T$ has the form

\begin{align*}
T = \prod_{\delta|20}\eta(\delta\tau)^{s_{\delta}}.
\end{align*}  Being a modular function over $\Gamma_0(20)$, we know that $s=(s_{\delta})_{\delta | 20}$ must satisfy the conditions (\ref{newman1a})--(\ref{newman1d}) of Newman's Theorem:

\begin{align*}
&\sum_{\delta | 20} s_{\delta} = 0,\\
&\sum_{\delta | 20} \delta s_{\delta} +24x_1 = 0, \\
&\sum_{\delta | 20} \frac{20}{\delta} s_{\delta} + 24x_2 = 0,\\
&\prod_{\delta | 20} \delta^{|s_{\delta}|} = x_3^2,
\end{align*} with $x_1, x_2, x_3\in\mathbb{Z}$.  What additional conditions are necessary for $T$?

Notice that $A\in\mathcal{M}(100)$, while $T, G_k\in\mathcal{M}(20)$.  Because $\mathcal{M}(20) \subseteq\mathcal{M}(100)$, we may take the product $A^iT^jG_k\in\mathcal{M}(100)$.  Then our $U_5$ operator maps $A^iT^jG_k\in\mathcal{M}(100)$ to

\begin{align*}
f^{(i,j)}_k &= U_5 \left( A^i T^j G_k \right)\\ &= \frac{1}{5}\sum_{r=0}^4 A^i\left( \frac{\tau + r}{5} \right) T^j\left( \frac{\tau + r}{5} \right) G_k\left( \frac{\tau + r}{5} \right)\in\mathcal{M}(20).
\end{align*}  We need to account for any possible poles of $f^{(i,j)}_k$, so that $T^{m}f^{(i,j)}_k\in\mathcal{M}^{\infty}(20)$ for sufficiently large $m\in\mathbb{Z}_{>0}$.  We will consider $A$ before functions $T, G_k$.

We now give a set of representatives for the cusps of $\Gamma_0(20)$, and for those of $\Gamma_0(100)$.  They may be calculated as in \cite[Lemma 5.3]{Radu1}:

\begin{align*}
\mathcal{C}(20) =&\left\{\frac{1}{20},\frac{1}{10},\frac{1}{5},\frac{1}{4},\frac{1}{2},1\right\},\\
\mathcal{C}(100) =&\left\{\frac{1}{100},\frac{1}{50},\frac{1}{25},\frac{1}{20},\frac{1}{10},\frac{3}{20}, \frac{1}{5}, \frac{1}{4}, \frac{3}{10}, \frac{7}{20},\frac{2}{5},\frac{9}{20},\frac{1}{2},\frac{3}{5},\frac{7}{10},\frac{4}{5},\frac{9}{10},1\right\}.
\end{align*}

In the first place, we have an exact form for $A$, which allows us to compute its zeros and poles exactly, via Ligozat's theorem (\ref{ligozat11}).  Doing so, and employing Theorem 3, reveals the following:

\begin{align*}
\mathrm{ord}_{1/100}^{(100)}(A) &= 1,\\
\mathrm{ord}_{1/50}^{(100)}(A) &= -5,\\
\mathrm{ord}_{1/25}^{(100)}(A) &= 4,\\
\mathrm{ord}_{1/4}^{(100)}(A) &= -1,\\
\mathrm{ord}_{1/2}^{(100)}(A) &= 5,\\
\mathrm{ord}_{1}^{(100)}(A) &= -4.
\end{align*}

In particular, $A$ has negative order (i.e., poles) at $1/50, 1/4, 1$.  Because $U_5$ sends $A$ to $\frac{1}{5}\sum_{r=0}^4 A^i\left( (\tau + r)/5 \right)$, we need to examine the possible rational numbers $\tau$ may approach so that $(\tau + r)/5$ approaches a rational number corresponding to the cusps at $1/50, 1/4, 1$.

In Table \ref{title2a} we take $\tau$ to approach an element of $\mathcal{C}(20)$.  In the process, $(\tau+r)/5$ will tend to a rational number for $r=0,1,2,3,4$.  We then take the element in $\mathcal{C}(100)$ representing the same cusp as $(\tau+r)/5$ through use of Theorem 3 above.  For example, as $\tau\rightarrow 1/10$, and for $r=3$, $(\tau+r)/5\rightarrow 31/50$.  However,  if we set $a_1/c_1 = 1/50\in\mathcal{C}(100)$ and $a/c = 31/50$, and take $m=31, n=0$, then the congruences of Theorem 3 are satisfied, so that for $\tau\rightarrow 1/10$ and $r=3$, we have the corresponding cusp $1/50$.

\begin{table}[!ht]
\begin{center}
\scalebox{0.8}{\begin{tabular}{ l | c  c  c  c  r }
 & & & $r$\\
\hline\\
Elements $a/c$ of $\mathcal{C}(20)$ Approached by $\tau$\ \ \           & $0$ & 1 & 2 & 3 & 4 \\
\hline \\
$\frac{1}{20}$          & $\frac{1}{100}$ & $\frac{1}{100}$ & $\frac{1}{100}$ & $\frac{1}{100}$ & $\frac{1}{100}$ \\ \\
$\frac{1}{10}$          & $\frac{1}{50}$ & $\frac{1}{50}$ & $\frac{1}{50}$ & $\frac{1}{50}$ & $\frac{1}{50}$ \\ \\
$\frac{1}{5}$            & $\frac{1}{25}$ & $\frac{1}{25}$ & $\frac{1}{25}$ & $\frac{1}{25}$ & $\frac{1}{25}$ \\ \\
$\frac{1}{4}$            & $\frac{1}{20}$ & $\frac{1}{4}$ & $\frac{9}{20}$ & $\frac{3}{20}$ & $\frac{7}{20}$ \\ \\
$\frac{1}{2}$            & $\frac{1}{10}$ & $\frac{3}{10}$ & $\frac{1}{2}$ & $\frac{7}{10}$ & $\frac{9}{10}$ \\  \\
1               & $\frac{1}{5}$ & $\frac{2}{5}$ & $\frac{3}{5}$ & $\frac{4}{5}$ & $1$ \\ \\
\end{tabular}
}
\caption{Elements of $\mathcal{C}(100)$ Approached by $\frac{\tau + r}{5}$}\label{title2a}
\end{center}
\end{table}

Notice that just three cusps over $\Gamma_0(20)$ (represented by $1, 1/2, 1/4$) correspond to 15 of the 18 cusps of $\Gamma_0(100)$.  The remaining three cusps of $\Gamma_0(20)$ ($1/5, 1/10, 1/20$) correspond bijectively to the remaining cusps over $\Gamma_0(100)$ ($1/25, 1/50, 1/100$).

We see that for $(\tau+r)/5$ to approach the cusps $1/50, 1/4, 1$, $\tau$ must approach $1/10, 1/4, 1$, respectively.

In other words, $U_5\left(A\right)$ has possible poles at the cusps $1/10, 1/4, 1$.  We therefore want our $T$ to have positive order at these cusps.  We therefore have the following system of inequalities that we know are necessary (but not yet sufficient) for $T = \prod_{\delta|20}\eta(\delta\tau)^{s_{\delta}}$:

\begin{align*}
&\frac{1}{24}\sum_{\delta | 20} \frac{\mathrm{gcd}(10,\delta)^2}{\delta} s_{\delta} \ge 1,\\
&\frac{5}{24}\sum_{\delta | 20} \frac{\mathrm{gcd}(4,\delta)^2}{\delta} s_{\delta} \ge 1,\\
&\frac{5}{6}\sum_{\delta | 20} \frac{\mathrm{gcd}(1,\delta)^2}{\delta} s_{\delta} \ge 1.\\
\end{align*}

Next we consider $G_k$ for $1\le k\le v$.  By our definition, we want $G_k\in\mathcal{M}^{\infty}(20)$, so that $G_k$ only has a pole at the cusp at $\infty$ with respect to $\Gamma_0(20)$.  Table \ref{title2b} below is analogous to Table \ref{title2a} but only considering the cusps of $\Gamma_0(20)$.

\begin{table}[!ht]
\begin{center}
\scalebox{0.8}{\begin{tabular}{ l | c  c  c  c  r }
 & & & $r$\\
\hline\\
Elements $a/c$ of $\mathcal{C}(20)$ Approached by $\tau$ \ \ \           & $0$ & 1 & 2 & 3 & 4 \\
\hline \\
$\frac{1}{20}$          & $\frac{1}{20}$ & $\frac{1}{20}$ & $\frac{1}{20}$ & $\frac{1}{20}$ & $\frac{1}{20}$ \\ \\
$\frac{1}{10}$          & $\frac{1}{10}$ & $\frac{1}{10}$ & $\frac{1}{10}$ & $\frac{1}{10}$ & $\frac{1}{10}$ \\ \\
$\frac{1}{5}$            & $\frac{1}{5}$ & $\frac{1}{5}$ & $\frac{1}{5}$ & $\frac{1}{5}$ & $\frac{1}{5}$ \\ \\
$\frac{1}{4}$            & $\frac{1}{20}$ & $\frac{1}{4}$ & $\frac{1}{20}$ & $\frac{1}{20}$ & $\frac{1}{20}$ \\ \\
$\frac{1}{2}$            & $\frac{1}{10}$ & $\frac{1}{10}$ & $\frac{1}{2}$ & $\frac{1}{10}$ & $\frac{1}{10}$ \\  \\
1               & $\frac{1}{5}$ & $\frac{1}{5}$ & $\frac{1}{5}$ & $\frac{1}{5}$ & $1$ \\ \\
\end{tabular}
}
\caption{Elements of $\mathcal{C}(20)$ Approached by $\frac{\tau + r}{5}$}\label{title2b}
\end{center}
\end{table}

Notice that the cusp at $\infty$ is represented in $\Gamma_0(20)$ by $1/20$, which may be approached as $\tau$ approaches the cusps $1/20, 1/4$ over $\Gamma_0(20)$.  Because we want $T$ to have a pole at $1/20$, we therefore only need to account for the additional possible pole at $1/4$, which we already accounted for.

Therefore, a function $T$ satisfying these three inequalities, together with (\ref{newman1a})--(\ref{newman1d}), will satisfy

\begin{align*}
T^{m}U_5\left(A^iT^jG_k\right)\in\mathcal{M}^{\infty}(20)
\end{align*} for $i=0,1, j\ge 0, 1\le k\le v$, with sufficiently large $m$.

Finally, there is the question of negative powers of $T$.  We know that because $T$ must have positive order at $1/10, 1/4, 1$, therefore $T^{-1}$ must have negative order at these cusps.  This means of course that $T$ must have positive order at any cusp representative $a/c$ such that

\begin{align*}
\frac{a/c + r}{5} = \frac{a + cr}{5c} \in\left\{ \frac{1}{10}, \frac{1}{4}, 1 \right\}.
\end{align*}  Examining our table above, it can quickly be seen that these values are approached as $\tau$ approaches the cusps at $1/10, 1/2, 1/4, 1$.  This induces another constraint: $T$ must have positive order at $1/2$.

We now have the additional inequality

\begin{align*}
&\frac{5}{24}\sum_{\delta | 20} \frac{\mathrm{gcd}(2,\delta)^2}{\delta} s_{\delta} \ge 1.
\end{align*}

We now have conditions for the behavior of $T$ at every cusp of $\Gamma_0(20)$ except for $\frac{1}{5}$.  Since $A, G_k$ do not have poles at $1/5$, we need only worry about $T$ and $T^{-1}$.  Suppose first that $T$ has positive order at $1/5$.  Then of course, $T^{-1}$ must have negative order at $1/5$.  Which cusps over $\Gamma_0(20)$ correspond to a potential pole at $1/5$?  Examining our table above, we see that the only possible poles induced would occur at $1/5, 1$.  Now $T$ already has positive order at $1$, as well as at $1/5$ by hypothesis.

Therefore, since the cusp at $1/5$ causes no problems whether $T$ has positive or zero order there, we do not need to induce any specific condition at the cusp (besides the nonnegative order of $T$).

\begin{align*}
&\frac{1}{6}\sum_{\delta | 20} \frac{\mathrm{gcd}(5,\delta)^2}{\delta} s_{\delta} \ge 0.
\end{align*}

We now have sufficient conditions from which to derive $T$, but we give one more mild condition for the sake of efficiency.  We clearly want $|\mathrm{ord}_{1/20}^{(20)}(T)|$ to be as small as possible.  We therefore take note of the fact that

\begin{align*}
\mathrm{ord}_{1/20}^{(20)}(T) = \frac{1}{24}\sum_{\delta | 20} \delta s_{\delta},
\end{align*} so that in our Newman system, $x_1 = -\mathrm{ord}_{1/20}^{(20)}(T)$.  We therefore add the additional tentative condition to our system:

\begin{align*}
x_1 = 1,
\end{align*} to search for the possibility that there exists an acceptable $T$ with $\mathrm{ord}_{1/20}^{(20)}(T) = -1$.  If our system contains no solution, then we must reset $x_1= 2$ and continue.

Our complete system, then, is 

\begin{align*}
&\sum_{\delta | 20} s_{\delta} = 0,\\
&\sum_{\delta | 20} \delta s_{\delta} +24x_1 = 0, \\
&\sum_{\delta | 20} \frac{20}{\delta} s_{\delta} + 24x_2 = 0,\\
&\prod_{\delta | 20} \delta^{|s_{\delta}|} = x_3^2,\\
&\frac{1}{10}\sum_{\delta | 20} \frac{\mathrm{gcd}(10,\delta)^2}{\delta} s_{\delta} \ge 1,\\
&\frac{1}{6}\sum_{\delta | 20} \frac{\mathrm{gcd}(5,\delta)^2}{\delta} s_{\delta} \ge 0,\\
&\frac{5}{24}\sum_{\delta | 20} \frac{\mathrm{gcd}(4,\delta)^2}{\delta} s_{\delta} \ge 1,\\
&\frac{5}{24}\sum_{\delta | 20} \frac{\mathrm{gcd}(2,\delta)^2}{\delta} s_{\delta} \ge 1,\\
&\frac{1}{6}\sum_{\delta | 20} \frac{\mathrm{gcd}(1,\delta)^2}{\delta} s_{\delta} \ge 1,\\
&x_1 = -\mathrm{ord}_{1/20}^{(20)}(T),
\end{align*} with $x_1, x_2, x_3\in\mathbb{Z}$.

This system allows us to obtain a vector $s$ that is optimal with respect to $x_1$.

We made use of the software package 4ti2 \cite{4ti2} to solve this system, and discovered the solution vector $s=(2,0,2,-2,8,-10)$, of minimal order $x_1 = 5$.  This gives us the function

\begin{align*}
T &= \frac{\eta(\tau)^2\eta(4\tau)^2\eta(10\tau)^{8}}{\eta(5\tau)^2\eta(20\tau)^{10}} = \frac{1}{q^5}\frac{(q;q)^2_{\infty}(q^4;q^4)^2_{\infty}(q^{10};q^{10})^{8}_{\infty}}{(q^5;q^5)^2_{\infty}(q^{20};q^{20})^{10}_{\infty}}.
\end{align*}  From here we may apply the algebra basis function of \cite[Section 2.1, Algorithm AB]{Radu} to construct the functions $G_k$.

\begin{theorem}

Let
\begin{align}
T &:= \frac{\eta(\tau)^2\eta(4\tau)^2\eta(10\tau)^{8}}{\eta(5\tau)^2\eta(20\tau)^{10}} = \frac{1}{q^5}\frac{(q;q)^2_{\infty}(q^4;q^4)^2_{\infty}(q^{10};q^{10})^{8}_{\infty}}{(q^5;q^5)^2_{\infty}(q^{20};q^{20})^{10}_{\infty}}\\
H &:= \frac{\eta(4\tau)\eta(5\tau)^5}{\eta(\tau)\eta(20\tau)^5} = \frac{1}{q^3}\frac{(q^4;q^4)_{\infty}(q^{5};q^{5})^{5}_{\infty}}{(q;q)_{\infty}(q^{20};q^{20})^{5}_{\infty}},\\
G &:= \frac{\eta(4\tau)^4\eta(10\tau)^2}{\eta(2\tau)^2\eta(20\tau)^4} = \frac{1}{q^2}\frac{(q^4;q^4)^4_{\infty}(q^{10};q^{10})^{2}_{\infty}}{(q^2;q^2)^2_{\infty}(q^{20};q^{20})^{4}_{\infty}}.
\end{align}  Then

\begin{align}
\mathcal{M}^{\infty}(20)_{\mathbb{Q}} = \left< 1, G_1, G_2, G_3, G_4 \right>_{\mathbb{Q}[T]},\label{fullMod1}
\end{align} with

\begin{align}
G_1 &= G,\\
G_2 &= H - G,\\
G_3 &= G^2,\\
G_4 &= (H - G)^2.
\end{align}  Moreover,

\begin{align}
U_5\left(A^iT^jG_k\right)\in\left< 1, G_1, G_2, G_3, G_4 \right>_{\mathbb{Z}[T,T^{-1}]},\label{Talready1}
\end{align} for all $(i,j,k)\in\{0,1\}\times\mathbb{Z}\times\{0,1,2,3,4\}$.

\end{theorem}

With $T$ derived, the algebra basis may be found with Radu's basis algorithm.  We prove its validity using the properties of the corresponding modular curve $\mathrm{X}_0(20)$, together with the Weierstrass gap theorem.

Notice that we restrict our coefficients to rational numbers, but that our theorem applies equally if we extend our field to the whole of $\mathbb{C}$.

\begin{proof}

Condition (\ref{Talready1}) was verified in the construction of $T$.  We are left to verify (\ref{fullMod1}).

Conditions (\ref{newman1a})--(\ref{newman1d}) can be quickly checked with respect to $G, H,$ and $T$, so that

\begin{align}
\mathcal{M}^{\infty}(20)_{\mathbb{Q}} \supseteq \left< 1, G_1, G_2, G_3, G_4 \right>_{\mathbb{Q}[T]}.\label{supmod1}
\end{align}  Let $f\in\mathcal{M}^{\infty}(20)_{\mathbb{Q}}$.  We want to prove that $f\in\left< 1, G_1, G_2, G_3, G_4 \right>_{\mathbb{Q}[T]}$.

With only one pole, $f$ has an expansion

\begin{align*}
f = \frac{b(-m_0)}{q^{m_0}} + \frac{b(-m_0+1)}{q^{m_0-1}} + ... + \frac{b(-1)}{q} + b(0) + \sum_{n=0}^{\infty} b(n)q^n,
\end{align*} with $b(n)\in\mathbb{Q}$ for all $n\ge -m_0$, and $b(-m_0)\neq 0$.

We can now apply the MC algorithm given in Section 3.1.  If we first assume that $m_0\neq 1$, then there exist $a,b\in\mathbb{Z}_{\ge 0}$ such that $m_0=5a +b$, and\\ $b\in\{0,2,3,4,6\}$.  Examining the orders of the functions $T, G_k$, we find that

\begin{align*}
-\mathrm{ord}_{\infty}^{(20)} (T) &= 5,\\
-\mathrm{ord}_{\infty}^{(20)} (G_1) &= 2,\\
-\mathrm{ord}_{\infty}^{(20)} (G_2) &= 3,\\
-\mathrm{ord}_{\infty}^{(20)} (G_3) &=4,\\
-\mathrm{ord}_{\infty}^{(20)} (G_4) &=6.
\end{align*}  We therefore have

\begin{align*}
-\mathrm{ord}_{\infty}^{(20)}\left(f_1\right) < m_0,
\end{align*} for

\begin{align*}
f_1 = f - \frac{b(m_0)}{\mathrm{LC}(T^a G_{k_1})}\cdot T^a G_{k_1} \in\mathcal{M}^{\infty}(20)_{\mathbb{Q}},
\end{align*} with some $k_1\in\{0,1,2,3,4\}$ (and taking $G_0=1$) such that $-\mathrm{ord}_{\infty}^{(20)} (G_{k_1}) = b$.

As described in the MC algorithm, we construct a sequence of functions $\mathcal{F}=\{f, f_1, f_2, ... \}$, each of which has a pole only at infinity, with $m_j := |\mathrm{ord}_{\infty}^{(20)}(f_{j})|$, and $m_{j+1}<m_j$ for all $j\ge 0$.  Membership is excluded if and only if within this sequence a function is produced with order exactly $-1$ at $\infty$.  If we can prove that such a function can never be produced, then our sequence of functions must ultimately have order $0$ at $\infty$, and membership is guaranteed.

Let us suppose that such a function does exist in our sequence, i.e., for some $M\in\mathbb{Z}_{\ge 0}$, $f_M\in\mathcal{F}$ has a pole only at $\infty$ and with order exactly $-1$.  In that case, $(f_M)^n$ will have order $-n$ for all $n\in\mathbb{Z}_{>0}$.  In other words, we can produce a function in $\mathcal{M}^{\infty}(20)$ with a pole only at $\infty$, and any order at that pole.

However, the functions of $\mathcal{M}^{\infty}(20)$ correspond bijectively to the functions of the modular curve $\mathrm{X}_0(20)$ with a pole only at $[\infty]$ \cite[Chapters 2, 3]{Diamond}.  This curve has genus 1 \cite[Chapter 4, Theorem 15]{Schoeneberg}, and the Weierstrass gap theorem \cite{Paule3} therefore requires that exactly one order must exist which cannot be assumed by any function over $\mathrm{X}_0(20)$ with a pole only at $[\infty]$.

But we just demonstrated that $f_M$ taken to positive powers may assume any order at $\infty$, and that we can therefore construct functions over $\mathrm{X}_0(20)$ with a single pole of any order.  We have a contradiction, and must therefore reject the hypothesis that such an $f_M$ is ever produced.

Because this is the only possible case in which membership fails, we must conclude that we can complete our reduction of $f$, so that

\begin{align*}
f\in\left< 1, G_1, G_2, G_3, G_4 \right>_{\mathbb{Q}[T]}.
\end{align*}  We then have

\begin{align}
\mathcal{M}^{\infty}(20)_{\mathbb{Q}} \subseteq \left< 1, G_1, G_2, G_3, G_4 \right>_{\mathbb{Q}[T]},\label{submod1}
\end{align} which, with (\ref{supmod1}), yields equality.
\end{proof}

\begin{corollary}
\begin{align*}
\mathcal{M}^{\infty}(20)_{\mathbb{Q}} = \left< \mathcal{E}^{\infty}(20) \right>_{\mathbb{Q}}
\end{align*}
\end{corollary}

\begin{proof}
\begin{align*}
\mathcal{M}^{\infty}(20)_{\mathbb{Q}} = \left< 1, G_1, G_2, G_3, G_4 \right>_{\mathbb{Q}[T]} \subseteq \left< \mathcal{E}^{\infty}(20) \right>_{\mathbb{Q}}\subseteq \mathcal{M}^{\infty}(20)_{\mathbb{Q}}.
\end{align*}
\end{proof}

Finally, we give the order of $T$ at its poles and zeros through (\ref{ligozat11}):

\begin{align*}
\mathrm{ord}_{1/20}^{(20)}(T) &= -5,\\
\mathrm{ord}_{1/10}^{(20)}(T) &= 1,\\
\mathrm{ord}_{1/5}^{(20)}(T) &= 0,\\
\mathrm{ord}_{1/4}^{(20)}(T) &= 1,\\
\mathrm{ord}_{1/2}^{(20)}(T) &= 1,\\
\mathrm{ord}_{1}^{(20)}(T) &= 2.
\end{align*}

\subsection{Powers of $T$}

We now give an outline for how to compute the powers $m\in\mathbb{Z}_{>0}$ so that 

\begin{align*}
T^{m}U_5\left(A^iT^jG_k\right)\in\mathcal{M}^{\infty}(20).
\end{align*}  It is clear, by our definition of $T$, that such a power must exist.

To begin, let us suppose that $h_1, h_2\in\mathcal{M}(100)$, and that $U_5\left(h_1\right), U_5\left(h_2\right)$ have poles which are canceled by the zeros of $T$.  In this case, nonnegative integers $m_1, m_2$ must exist such that

\begin{align*}
T(\tau)^{m_1}\cdot U_5\left(h_1\right) &= U_5\left(T(5\tau)^{m_1}h_1(\tau)\right)\in\mathcal{M}^{\infty}(20),\\
T(\tau)^{m_2}\cdot U_5\left(h_2\right) &= U_5\left(T(5\tau)^{m_2}h_2(\tau)\right)\in\mathcal{M}^{\infty}(20).
\end{align*}

Given $i\in\{1,2\}$, one way of ensuring that $U_5\left(T(5\tau)^{m_i}h_i(\tau)\right)$ has no poles over $\Gamma_0(20)$ besides that at $1/20$ is by ensuring that $T(5\tau)^{m_i}h_i(\tau)$ has no poles over $\Gamma_0(100)$ other than those which will manifest in $\Gamma_0(20)$ at $1/20$.  That is, we need to ensure that $T(5\tau)^{\xi_i}h_i(\tau)\in\mathcal{M}^{\infty}(100)$.  Any cusp of $\Gamma_0(100)$ other than $1/100$ can be approached by $(\tau+r)/5$ as $\tau$ approaches a cusp not represented by $1/20$ (see Table \ref{title2a}).

In this case, if we wish to examine $U_5\left(h_1\cdot h_2\right)$, we may note that

\begin{align*}
T(\tau)^{m_1+m_2}\cdot U_5\left(h_1\cdot h_2\right) &= U_5\left(T(5\tau)^{m_1+m_2}h_1(\tau)\cdot h_2(\tau)\right)\\
&=U_5\left(T(5\tau)^{m_1}h_1(\tau)\cdot T(5\tau)^{m_2}h_2(\tau)\right).
\end{align*}  Therefore, if $T(5\tau)^{m_1}h_1(\tau)$ and $T(5\tau)^{m_2}h_2(\tau)$ are both members of $\mathcal{M}^{\infty}(100)$, then their product must be as well.  But this means that 

\begin{align*}
T(\tau)^{m_1+m_2}\cdot U_5\left(h_1\cdot h_2\right)\in\mathcal{M}^{\infty}(20).
\end{align*}

Therefore, if we have sufficient powers of $T$ to push two functions $U_5(h_1), U_5(h_2)$ into $\mathcal{M}^{\infty}(20)$, then we need only add the powers together to have a sufficient power of $T$ to push $U_5(h_1\cdot h_2)$ into $\mathcal{M}^{\infty}(20)$. 

So in order to work out sufficient powers of $T$ to push $U_5\left(A^iT^jG_k\right)$ into $\mathcal{M}^{\infty}(20)$, it is necessary only to know the sufficient powers of $T$ for 

\begin{align*}
U_5\left(A\right),\ U_5\left(T\right),\ U_5\left(T^{-1}\right),\ U_5\left(G_k\right),\ 1\le k\le 4.
\end{align*}  

Let us suppose that the most optimal powers of $T$ for this purpose are

\begin{align*}
m_A,\ m_{+t},\ m_{-t},\ m_k,\ 1\le k\le 4,
\end{align*} respectively.  In that case, each of these powers will correspond to the highest-order pole of the corresponding function over $\Gamma_0(100)$ (excluding the cusp at $\infty$, of course).

Notice that $G_3=G^2$, $G_4 = G_2^2$, and $G_2 = H - G$.  Therefore, the orders for $U_5(G_3), U_5(G_4)$, respectively, will simply be double the orders of $U_5(G),U_5(G_2)$, respectively.  So we need only examine the orders of $G, G_2$.  Also, we know that $G_2 = H - G$, so that we need to examine the orders of $G$ and $H$.  That is, we can compute $m_2, m_3, m_4$ using only the necessary powers for 

\begin{align*}
U_5\left(G\right),\ U_5\left(H \right).
\end{align*}  Let us refer to the necessary power for $U_5\left(H \right)$ as $m_H$.  Then we have\\ $m_2 = \max\{m_1, m_H\}$, $m_3 = 2m_1$, $m_4 = 2m_2$.

Finally, supposing that $m_{t} = m_{\mathrm{sign}(j)t}$, then for $U_5\left(A^iT^jG_k\right)$, we have

\begin{align}
m(i,j,k) = i\cdot m_A + j\cdot m_{t}+m_k.
\end{align}

\subsubsection{Powers For $A$, $T$, $T^{-1}$, $G$, $H$}

We begin with $m_A$ as our principal example.  We know that for $m_A$ sufficiently large, we have

\begin{align*}
T(\tau)^{m_A}\cdot U_5\left(A(\tau)\right)\in\mathcal{M}^{\infty}(20).
\end{align*}  But notice that we can rewrite

\begin{align*}
T(\tau)^{m_A}\cdot U_5\left(A(\tau)\right) &= U_5\left(T(5\tau)^{m_A}A(\tau)\right).
\end{align*}  As covered in the beginning of the section, we need to ensure that $T(5\tau)^{m_A}A(\tau)$ only have a pole at the cusp represented by $1/100$.

Of course, 

\begin{align*}
\mathrm{ord}_{a/c}^{(100)}\left( T(5\tau)^{m_A} \right) = m_A\cdot \mathrm{ord}_{a/c}^{(100)}\left( T(5\tau) \right).
\end{align*}  With this in mind, in Table \ref{title2c} we examine the order of 

\begin{align*}
T(5\tau)^{m_A}A(\tau),\ T(5\tau)^{m_{+t}}T(\tau),\ T(5\tau)^{m_{-t}}T(\tau)^{-1}
\end{align*} at the cusps over $\Gamma_0(100)$ using (\ref{ligozat11}) once more.  In Table \ref{title2d} we examine the orders of 

\begin{align*}
T(5\tau)^{m_1}G(\tau),\ T(5\tau)^{m_H}H(\tau).
\end{align*}

\begin{table}
\begin{center}
\scalebox{0.8}{\begin{tabular}{ l | c c  r }
 &  & $f$ & \\
\hline\\
$\frac{a}{c}\in\mathcal{C}(100)$\ \ \           & $T(5\tau)^{m_A}A(\tau)$ & $T(5\tau)^{m_{+t}}T(\tau)$ & $T(5\tau)^{m_{-t}}T(\tau)^{-1}$ \\
\hline \\
$\frac{1}{100}$          & $1-25m_A$  & $-5-25m_{+t}$ & $5-25m_{-t}$ \\ \\
$\frac{1}{50}$          & $-5+5m_A$  & $1+5m_{+t}$ & $-1+5m_{-t}$ \\ \\
$\frac{1}{25}$            & $4$  & $0$ & $0$ \\ \\
$\frac{1}{20}$            & $m_A$  & $-5+m_{+t}$ & $5+m_{-t}$ \\ \\
$\frac{1}{10}$            & $m_A$  & $1+m_{+t}$ & $-1+m_{-t}$ \\  \\
$\frac{3}{20}$            & $m_A$  & $-5+m_{+t}$ & $5+m_{-t}$ \\  \\
$\frac{1}{5}$            & $2m_A$  & $2m_{+t}$ & $2m_{-t}$ \\  \\
$\frac{1}{4}$            & $-1+m_A$  & $5+m_{+t}$ & $-5+m_{-t}$ \\  \\
$\frac{3}{10}$            & $m_A$  & $1+m_{+t}$ & $-1+m_{-t}$ \\  \\
$\frac{7}{20}$            & $m_A$  & $-5+m_{+t}$ & $5+m_{-t}$ \\  \\
$\frac{2}{5}$            & $2m_A$  & $2m_{+t}$ & $2m_{-t}$ \\  \\
$\frac{9}{20}$            & $m_A$  & $-5+m_{+t}$ & $5+m_{-t}$ \\  \\
$\frac{1}{2}$            & $5+m_A$  & $5+m_{+t}$ & $-5+m_{-t}$ \\  \\
$\frac{3}{5}$            & $2m_A$  & $2m_{+t}$ & $2m_{-t}$ \\  \\
$\frac{7}{10}$            & $m_A$  & $1+m_{+t}$ & $-1+m_{-t}$ \\  \\
$\frac{4}{5}$            & $2m_A$  & $2m_{+t}$ & $2m_{-t}$ \\  \\
$\frac{9}{10}$            & $m_A$   & $1+m_{+t}$ & $-1+m_{-t}$ \\  \\
1                              & $-4+2m_A$  & $10+2m_{+t}$ & $-10+2m_{-t}$ \\ \\
\end{tabular}
}
\caption{$\mathrm{ord}_{a/c}^{(100)}(f)$ for $a/c\in\mathcal{C}(100)$}\label{title2c}
\end{center}
\end{table}

\begin{table}
\begin{center}
\scalebox{0.8}{\begin{tabular}{ l | c  r }
$\mathrm{ord}_{a/c}^{(100)}(f)$ &  $f$ & \\
\hline\\
$\frac{a}{c}\in\mathcal{C}(100)$\ \ \           & $T(5\tau)^{m_1}G(\tau)$ & $T(5\tau)^{m_{H}}H(\tau)$ \\
\hline \\
$\frac{1}{100}$          & $-2-25m_1$  & $-3-25m_2$ \\ \\
$\frac{1}{50}$          & $5m_1$  & $5m_2$ \\ \\
$\frac{1}{25}$            & $0$  & $3$ \\ \\
$\frac{1}{20}$            & $-2+m_1$  & $-3+m_2$ \\ \\
$\frac{1}{10}$            & $m_1$  & $m_2$ \\  \\
$\frac{3}{20}$            & $-2+m_1$  & $-3+m_2$ \\  \\
$\frac{1}{5}$            & $2m_1$  & $3+2m_2$ \\  \\
$\frac{1}{4}$            & $10+m_1$  & $m_2$ \\  \\
$\frac{3}{10}$            & $m_1$  & $m_2$ \\  \\
$\frac{7}{20}$            & $-2+m_1$  & $-3+m_2$ \\  \\
$\frac{2}{5}$            & $2m_1$  & $3+2m_2$ \\  \\
$\frac{9}{20}$            & $-2+m_1$  & $-3+m_2$ \\  \\
$\frac{1}{2}$            & $m_1$  & $m_2$ \\  \\
$\frac{3}{5}$            & $2m_1$  & $3+2m_2$ \\  \\
$\frac{7}{10}$            & $m_1$  & $m_2$ \\  \\
$\frac{4}{5}$            & $2m_1$  & $3+2m_2$ \\  \\
$\frac{9}{10}$            & $m_1$   & $m_2$ \\  \\
1                              & $2m_1$  & $2m_2$ \\ \\
\end{tabular}
}
\caption{$\mathrm{ord}_{a/c}^{(100)}(f)$ for $a/c\in\mathcal{C}(100)$}\label{title2d}
\end{center}
\end{table}

As before, we find possible poles at $1/100, 1/50, 1/4, 1$.  We of course do not worry about the pole at $1/100$.  For the cusps at $1/50, 1/4$, we only need $m_A\ge 1$.  Finally, for the cusp at $1$, we need $m_A\ge 2$.  This gives us our best possible value: $m_A = 2$.

Similarly, we have $m_{+t} = 5$, $m_{-t} = 5$, and therefore that $m_t = 5$.

We also have $m_1 = 2$, $m_H = 3$.  Acknowledging that $m_2 = \max\{m_1, m_H\} = 3$, we finally have

\begin{align*}
m_2 = 3,\ m_3 = 4,\ m_4 = 6.
\end{align*}

\subsubsection{Complete Formula}

Putting everything together, we now have the following formula:

\begin{theorem}
For any $(i,j,k)\in\{0,1\}\times\mathbb{Z}\times\{0,1,2,3,4\}$, we have

\begin{align*}
T^{m(i,j,k)}U_5\left(A^iT^jG_k\right)\in\mathcal{M}^{\infty}(20),
\end{align*} with
\begin{align*}
m(i,j,k) = 2\cdot i + 5\cdot j + m_k.
\end{align*} and $m_1 = 2,\ m_2 = 3,\ m_3 = 4,\ m_4 = 6$.

\end{theorem}

\section{The Andrews--Sellers Conjecture}

A similar technique may be brought to bear on a famous conjecture, now proven by Paule and Radu \cite{Paule}.

\begin{theorem}
Let
\begin{align}
\mathrm{C}\Phi_2(q) := \sum_{n=0}^{\infty}c\phi_2(n)q^n = \frac{(q^2;q^2)^5_{\infty}}{(q;q)^4_{\infty}(q^4;q^4)^2_{\infty}}.
\end{align}  If $12n\equiv 1\pmod{5^{\alpha}}$, then $c\phi_2(n)\equiv 0\pmod{5^{\alpha}}$.
\end{theorem}  James Sellers conjectured this family of congruences in 1994, but substantial direct evidence for its validity was not gathered before 2001, when Eichhorn and Sellers proved the first four cases.  Their approach relied on recurrences given by a modular equation, and the total necessary calculations took place in 147 hours with a 600 MHz Pentium III Processor \cite[Section 3]{Eichhorn}.  Our approach allows us to check the first five cases with a 2.6 GHz Intel Processor in less than 2 hours.

We begin, as before, by defining generating functions over $c\phi_2(n)$, in which $n$ follows the necessary congruence condition:

\begin{align*}
L_0 &:= 1,\\
L_{\alpha} &:= \Phi_{\alpha}\cdot \sum_{12n\equiv 1\bmod{5^{\alpha}}} c\phi_2(n)q^{\left\lfloor n/5^{\alpha} \right\rfloor}.
\end{align*}  We define our function $A$ and operators $U^{\alpha}$ with respect to $\mathrm{C}\Phi_2$:

\begin{align*}
A := q^2\cdot\frac{\mathrm{C}\Phi_2(q)}{\mathrm{C}\Phi_2(q^{25})},
\end{align*}

\begin{align*}
U^{(0)}(f) &:= U_5 \left(A\cdot f\right),\\
U^{(1)}(f) &:= U_5(f),\\
U^{(\alpha)}(f) &:= U^{(\alpha\bmod{2})}(f),
\end{align*}

\begin{align*}
L_{2\alpha-1} &= U^{(0)} \left(L_{2\alpha-2}\right),\text{ and } L_{2\alpha} = U^{(1)} \left(L_{2\alpha-1}\right).
\end{align*}  This very quickly yields the following functions for $\Phi_{\alpha}$:

\begin{align*}
\Phi_{2\alpha - 1} &= \frac{q}{\mathrm{C}\Phi_2(q^5)}, \text{ and } \Phi_{2\alpha} = \frac{q}{\mathrm{C}\Phi_2(q)}.
\end{align*}

We want to verify 5-adic convergence for $(L_{\alpha})_{\alpha\ge 0}$.  Fortunately for us, $(L_{\alpha})_{\alpha\ge 0}\subseteq\mathcal{M}(20)$, as in the previous case, and $A$ possesses poles at the same cusps.  We may therefore employ the basis previously derived, and give only a slight modification to our algorithm:

To check the Andrews--Sellers conjecture for $0\le\alpha\le B$, begin by defining 
\begin{align*}
L_0^{(B)} &:= 1,\\
L_{\alpha}^{(B)} &:= U_5^{(\alpha-1)}(L_{\alpha-1}^{(B)}) \pmod{5^B} = \displaystyle\sum_{\substack{j\in\mathbb{Z},\\ 0\le k\le v}} c_{\alpha,j,k}T^{j}G_k,
\end{align*} with $0\le c_{\alpha,j,k}<5^A$ for all $\alpha, j, k$.

\begin{enumerate}
\item Begin with $\alpha=0$, $v_0=0$, and $V = \{ v_0 \}$.
\item Expand $L_{\alpha}^{(B)}$ into $\left< 1, G_1, ..., G_v \right>_{\mathbb{Z}[T,T^{-1}]}$: $L_{\alpha}^{(B)} = \displaystyle\sum_{\substack{j\in\mathbb{Z},\\ 0\le k\le v}} c_{\alpha,j,k}T^{j}G_k$.
\item Expand $U_5^{(\alpha)}(L_{\alpha}^{(B)}) = \displaystyle\sum_{\substack{j\in\mathbb{Z},\\ 0\le k\le v}} c_{\alpha,j,k}U_5^{(\alpha)}(T^{j}G_k)$.
\item Reduce $U_5^{(\alpha)}(L_{\alpha}^{(B)})\pmod{5^{B}}$ to get $L_{\alpha+1}^{(B)} = \displaystyle\sum_{\substack{j\in\mathbb{Z},\\ 0\le k\le v}} c_{\alpha+1,j,k}T^{j}G_k$.
\item Let $v_{\alpha+1}$ be the maximal power of 5 (up to $B$) dividing each nonzero $c_{\alpha +1,j,k}$.
\item Set $V = V\cup \{v_{\alpha+1}\}$.
\item Set $\alpha = \alpha+1$, and repeat.
\item Continue until $\alpha=B$.
\item If $v_{\alpha}=\alpha$ for $0\le \alpha\le B$, then we have verified our conjecture for the first $B$ cases.  Otherwise, the conjecture fails.
\end{enumerate}

Using this algorithm, we were able to verify the theorem for $0\le\alpha\le 5$ in 1 hour, 45 minutes.

\section{A More General Algorithm}

With these examples, we can now formulate a more general approach to our problems.  This is by no means comprehensive, but serves rather as a guide for how families of congruences can be studied from a large class of generating functions.

We now define an integer $M\in\mathbb{Z}_{>0}$, and an integer-valued vector $r=(r_{\delta})_{\delta | M}$ indexed over the divisors of $M$.  From this, we can define an arithmetic sequence with the generating function

\begin{align}
\mathcal{G}(q) := \prod_{\delta | M} (q^{\delta};q^{\delta})_{\infty}^{r_{\delta}} = \sum_{n=0}^{\infty} a(n)q^n.
\end{align}  Let us take a prime $\ell > 3$.  For simplicity, we will also take the assumption that

\begin{align*}
0\le -\sum_{\delta | M} \delta r_{\delta}\le \frac{24}{\ell+1}.
\end{align*}  From here, define

\begin{align*}
A(q) :=& q^{(1-\ell^2)\sum_{\delta | M}\delta r_{\delta}/24}\frac{\mathcal{G}(q)}{\mathcal{G}(q^{\ell^2})}.
\end{align*}  Let us set $N=\ell\cdot M$.  In this case, $A(q)$ satisfies the conditions (\ref{newman1a})--(\ref{newman1d}), and

\begin{align*}
A(q)\in\mathcal{M}(\ell^2\cdot M) = \mathcal{M}(\ell\cdot N).
\end{align*}

Next, we make the assumption that

\begin{align}
\mathcal{M}(N)_{\mathbb{Q}} = \left <\mathcal{E}(N)\right >_{\mathbb{Q}}.\label{newmanN}
\end{align}  That is, the space of all modular functions over $\Gamma_0(N)$ with rational coefficients is equal to the space generated by eta quotients over $\Gamma_0(N)$ with rational coefficients.

This condition was conjectured by Newman for all composite $N\in\mathbb{Z}_{>0}$ \cite[Section 8]{Newman1}.  In its original form, the conjecture no longer stands \cite[Section 3.3]{Radu}, though one of the authors has made a modification to the conjecture \cite[Conjecture 9.4]{Paule2}.  Very likely, our method may be extended to include modular curves in which (\ref{newmanN}) fails.  For the time being, we take it as true.

From here, we define the operators 

\begin{align*}
U^{(\alpha)}: \mathcal{M}(\ell\cdot N) \rightarrow \mathcal{M}(N),\ \alpha\in\mathbb{Z}_{\ge 0}
\end{align*} by

\begin{align*}
U^{(0)}\left(f\right) &:= U_\ell \left(A\cdot f\right),\\
U^{(1)}\left(f\right) &:= U_\ell\left(f\right),\\
U^{(\alpha)}\left(f\right) &:= U^{(\alpha\bmod{2})}\left(f\right).
\end{align*}  If we also define

\begin{align*}
\Phi_{2\alpha - 1} &:= \frac{q}{\mathcal{G}(q^{\ell})}, \text{ and } \Phi_{2\alpha}:= \frac{q}{\mathcal{G}(q)},
\end{align*} and set

\begin{align*}
L_0:=1,
\end{align*} then we define a sequence of functions $\mathcal{L} = (L_{\alpha})_{\alpha\ge 0}$ in which

\begin{align*}
L_{\alpha+1} &= U^{(\alpha)} \left(L_{\alpha}\right), \text{ and}
\end{align*}

\begin{align*}
L_{\alpha} = \Phi_{\alpha}\cdot \sum_{n\in C_{\ell,\alpha}} a(n)q^{\left\lfloor n/\ell^{\alpha} \right\rfloor},
\end{align*} with $C_{\ell, \alpha}$ a set of arithmetic progressions, with bases of the form $\ell^{\alpha}$.  In particular, for $\alpha=1$, we have

\begin{align*}
C_{\ell,1} = \left\{ \ell\cdot n +\ell + \frac{(\ell^2- 1)}{24}\sum_{\delta | M} \delta r_{\delta} : n\in\mathbb{Z}_{\ge 0}  \right\}.
\end{align*}

Suppose we suspect a family of congruences for $C_{\ell,\alpha}$.  That is, we believe that $\mathcal{L}$ is $\ell$-adically convergent to 0, and that we have a suspected pattern to the convergence.

From here, we define $\mathcal{C}(N)$ as a complete set of representatives for the cusps of $\Gamma_0(N)$, and similarly for $\mathcal{C}(\ell\cdot N)$.  We now must construct an appropriate algebra basis,

\begin{align*}
\left< \mathcal{E}^{\infty}(N) \right>_{\mathbb{Q}} = \left< 1, g_1, g_2, ..., g_v \right>_{\mathbb{Q}[t]},
\end{align*} such that

\begin{align*}
U_{\ell}\left(A^it^jg_k\right)\in\left< 1, g_1, g_2, ..., g_v \right>_{\mathbb{Q}[t,t^{-1}]},
\end{align*} for all $(i,j,k)\in\{0,1\}\times\mathbb{Z}\times\{0,1,...,v\}$.

We begin with the derivation of $t$.  As in the case of $\Gamma_0(20)$, we can give a system of equations and inequalities by which such a $t$ can be derived.  Let

\begin{align*}
t = \prod_{\delta | N}\eta(\delta\tau)^{w_{\delta}},
\end{align*} with $w:=(w_{\delta})_{\delta | N}$ an integer-valued vector.  We begin again with (\ref{newman1a})--(\ref{newman1d}):

\begin{align*}
&\sum_{\delta | N} w_{\delta} = 0,\\
&\sum_{\delta | N} \delta w_{\delta} +24x_1 = 0, \\
&\sum_{\delta | N} \frac{N}{\delta} w_{\delta} + 24x_2 = 0,\\
&\prod_{\delta | N} \delta^{|w_{\delta}|} = x_3^2,
\end{align*} with $x_1, x_2, x_3\in\mathbb{Z}$.

We now consider the poles of $A$.  Define $\mathcal{P}_{\ell\cdot N}(A)$ as a set of representatives of cusps in $\mathcal{C}(\ell\cdot N)$ for which $A$ possesses a pole.  Then define $\mathcal{P}(A)$ as

\begin{align*}
\mathcal{P}(A) &:= \left\{ \frac{a}{c}\in\mathcal{C}(N):  \frac{a + cr}{c\cdot\ell}\in\Gamma_0(N)\frac{a'}{c'},\text{ for some } \frac{a'}{c'}\in\mathcal{P}_{\ell\cdot N}(A),\ r\in\{0,1,...,\ell-1\} \right\}.
\end{align*}  We add to our system the inequalities

\begin{align*}
&\frac{N}{24\mathrm{gcd}(c^2,N)}\sum_{\delta | N} \frac{\mathrm{gcd}(c,\delta)^2}{\delta} w_{\delta} > 0, \text{ for all $a/c\in\mathcal{P}(A)$}.
\end{align*}

Now we consider the poles of $t,g_k$, $1\le k\le v$.  Over $\Gamma_{0}(N)$, they only have a pole at $\frac{1}{N}$.  Over $\Gamma_0(\ell\cdot N)$, however, $t,g_k$ will have possible poles for any cusp represented by $\frac{a'}{N}$, $\mathrm{gcd}(a,N)=1$.  

Let $\mathcal{P}(g)$ be defined as 

\begin{align*}
\mathcal{P}(g) &:= \left\{ \frac{a}{c}\in\mathcal{C}(N):  \frac{a + cr}{c\cdot\ell} =\frac{a'}{N},\ \mathrm{gcd}(a',N)=1, \text{for some } r\in\{0,1,...,\ell-1\}\right\}.
\end{align*}  We now have the additional set of inequalities

\begin{align*}
&\frac{N}{24\mathrm{gcd}(c^2,N)}\sum_{\delta | N} \frac{\mathrm{gcd}(c,\delta)^2}{\delta} w_{\delta} > 0, \text{ for all $\frac{a}{c}\in\mathcal{P}(g)$}.
\end{align*}  Finally, we examine $t^{-1}$.  Let

\begin{align*}
\mathcal{P}' = \mathcal{C}(N)\backslash (\mathcal{P}(A)\cup\mathcal{P}(g)),
\end{align*} and let $\frac{a}{c}\in\mathcal{P}'$.  Consider

\begin{align*}
\mathcal{P}_{a,c} := \bigg\{&\frac{a'}{c'}\in\mathcal{C}(N): \frac{a + cr}{c\cdot\ell}\in\Gamma_0(N)\frac{a'}{c'}\text{ for some } r\in\{0,1,...,\ell-1\}\bigg\},
\end{align*} and define

\begin{align*}
\mathcal{P}_0 ' &:= \left\{\frac{a}{c}\in P' : \mathcal{P}_{a,c}\subseteq \mathcal{P}(A)\cup\mathcal{P}(g)\cup\left\{\frac{a}{c}\right\}\right\},\\
\mathcal{P}_1 ' &:= \mathcal{P}'\backslash \mathcal{P}_0 '.
\end{align*}

If $\frac{a}{c}\in \mathcal{P}_0'$, then we need establish no condition beyond

\begin{align*}
&\frac{N}{24\mathrm{gcd}(c^2,N)}\sum_{\delta | N} \frac{\mathrm{gcd}(c,\delta)^2}{\delta} w_{\delta} \ge 0.
\end{align*}  If $\frac{a}{c}\in \mathcal{P}_1'$, we will have to decide which remaining cusps deserve positive order, and which deserve zero order.  Alternatively, we may simply set

\begin{align*}
&\frac{N}{24\mathrm{gcd}(c^2,N)}\sum_{\delta | N} \frac{\mathrm{gcd}(c,\delta)^2}{\delta} w_{\delta} = 0.
\end{align*}  This is not perfectly optimal, but gives us a complete set of equations and inequalities:

To summarize, we let $\mathcal{C}(N)$ be a complete set of representatives for the cusps of $\Gamma_0(N)$, and likewise for $\mathcal{C}(\ell\cdot N)$.  Let $\mathcal{P}_{\ell\cdot N}(A)\subseteq\mathcal{C}(\ell\cdot N)$ be the set of representatives of cusps for which $A$ possesses a pole.  Let

\begin{align*}
\mathcal{P}(A) &= \left\{ \frac{a}{c}\in\mathcal{C}(N):  \frac{a + cr}{c\cdot\ell}\in\Gamma_0(N)\frac{a'}{c'},\text{ for some } \frac{a'}{c'}\in\mathcal{P}_{\ell\cdot N}(A),\ r\in\{0,1,...,\ell-1\} \right\},\\
\mathcal{P}(g) &= \left\{ \frac{a}{c}\in\mathcal{C}(N):  \frac{a + cr}{c\cdot\ell} =\frac{a'}{N},\ \mathrm{gcd}(a',N)=1,\text{for some } r\in\{0,1,...,\ell-1\}\right\},\\
\mathcal{P}' &= \mathcal{C}(N)\backslash (\mathcal{P}(A)\cup\mathcal{P}(g)),\\
\mathcal{P}_0 ' &= \left\{\frac{a}{c}\in P' : \mathcal{P}_{a,c}\subseteq \mathcal{P}(A)\cup\mathcal{P}(g)\cup\left\{\frac{a}{c}\right\}\right\},\\
\mathcal{P}_1 ' &= \mathcal{P}'\backslash \mathcal{P}_0 '.
\end{align*}  For some $n_0\in\mathbb{Z}_{>0}$, define the system $W(n_0)$ by:

\begin{align*}
W(n_0):&\\
&\sum_{\delta | N} w_{\delta} = 0,\\
&\sum_{\delta | N} \delta w_{\delta} +24x_1 = 0, \\
&\sum_{\delta | N} \frac{N}{\delta} w_{\delta} + 24x_2 = 0,\\
&\prod_{\delta | N} \delta^{|w_{\delta}|} = x_3^2,\\
&\frac{N}{24\mathrm{gcd}(c^2,N)}\sum_{\delta | N} \frac{\mathrm{gcd}(c,\delta)^2}{\delta} w_{\delta} > 0 \text{ for all $\frac{a}{c}\in\mathcal{P}(A)$},\\
&\frac{N}{24\mathrm{gcd}(c^2,N)}\sum_{\delta | N} \frac{\mathrm{gcd}(c,\delta)^2}{\delta} w_{\delta} > 0 \text{ for all $\frac{a}{c}\in\mathcal{P}(g)$},\\
&\frac{N}{24\mathrm{gcd}(c^2,N)}\sum_{\delta | N} \frac{\mathrm{gcd}(c,\delta)^2}{\delta} w_{\delta} \ge 0 \text{ for all $\frac{a}{c}\in\mathcal{P}_0'$},\\
&\frac{N}{24\mathrm{gcd}(c^2,N)}\sum_{\delta | N} \frac{\mathrm{gcd}(c,\delta)^2}{\delta} w_{\delta} = 0 \text{ for all $\frac{a}{c}\in\mathcal{P}_1'$},\\
&x_1 = n_0.
\end{align*}  We first attempt to solve $W(1)$, i.e., to find a satisfactory $t$ such that $\mathrm{ord}_{1/20}^{(20)}\{t\}=-1$.  If no solution exists, we take $n_0 = n_0 + 1$ and repeat, until a solution is found.  This minimizes the order of $t$ at $\infty$ with respect to our system.

With $t$ defined, we use the basis algorithm of \cite[Section 2, Algorithm AB]{Radu} to produce the complete basis for $\left< \mathcal{E}^{\infty}(N) \right>_{\mathbb{Q}}$, which finally yields

\begin{align*}
\mathcal{M}^{\infty}(N) = \left< \mathcal{E}^{\infty}(N) \right>_{\mathbb{Q}} = \left< 1, g_1, ..., g_v \right>_{\mathbb{Q}[t]}.
\end{align*}

We now use (\ref{ligozat11}) to compute

\begin{align*}
\mathrm{ord}_{a/c}^{(\ell\cdot N)} \left(  t\left( \ell\cdot\tau \right)^{m_{f}}\left( f(\tau) \right) \right) = m_{f}\cdot \mathrm{ord}_{a/c}^{(\ell\cdot N)} \left(  t\left( \ell\cdot\tau \right)\right) + \mathrm{ord}_{a/c}^{(\ell\cdot N)}\left( \left( f(\tau) \right) \right),
\end{align*} with

\begin{align*}
f\in\left\{ A, t, t^{-1}, g_1, ..., g_v \right\}.
\end{align*}

We now compute the minimal $m_f$ such that

\begin{align*}
m_{f}\cdot \mathrm{ord}_{a/c}^{(\ell\cdot N)} \left(  t\left( \ell\cdot\tau \right)\right) + \mathrm{ord}_{a/c}^{(\ell\cdot N)}\left( \left( f(\tau) \right) \right) \ge 0.
\end{align*}  Finally, we quickly define

\begin{align*}
m_0:= m_0(j) &= \begin{cases}
    m_{t},       & j>0 \\
    -m_{t^{-1}}, & j<0
  \end{cases},\\
m_k &:= m_{g_k},\ 1\le k\le v. 
\end{align*}

This gives us the following:

\begin{align*}
t^{m(i,j,k)}\cdot U_{\ell}\left(A^it^jg_k\right)\in\left< \mathcal{E}^{\infty}(N) \right>_{\mathbb{Z}},
\end{align*} with

\begin{align*}
m(i,j,k) := i\cdot m_A + j\cdot m_0(j) + m_k.
\end{align*}

Now for any $B\in\mathbb{Z}_{>0}$, define

\begin{align*}
L_0^{(B)} &:= 1,\\
L_1^{(B)} &:= \displaystyle\sum_{\substack{j\in\mathbb{Z},\\ 0\le k\le v}} c_{1,j,k}t^{j}g_k,\\
L_{\alpha}^{(B)} &:= U^{(\alpha-1)}\left(L_{\alpha-1}^{(B)}\right) = \displaystyle\sum_{\substack{j\in\mathbb{Z},\\ 0\le k\le v}} c_{\alpha,j,k}t^{j}g_k,
\end{align*} where $c_{\alpha,j,k}$ is the coefficient of $t^{j}g_k$, reduced modulo $\ell^B$.

We now give the steps for examining $L_{\alpha}$ for $0\le\alpha\le B$ for possible divisibility by powers of $\ell$ (up to $\ell^B$):

\begin{enumerate}
\item Begin with $\alpha=0$, $v_0=0$, and $V = \{ v_0 \}$.
\item Expand $L_{\alpha}^{(B)}$ into $\left< 1, g_1, ..., g_v \right>_{\mathbb{Z}[t,t^{-1}]}$: $L_{\alpha}^{(B)} = \displaystyle\sum_{\substack{j\in\mathbb{Z},\\ 0\le k\le v}} c_{\alpha,j,k}t^{j}g_k$.
\item Compute $U^{(\alpha)}\left(L_{\alpha}^{(B)}\right) = \displaystyle\sum_{\substack{j\in\mathbb{Z},\\ 0\le k\le v}} c_{\alpha,j,k}U^{(\alpha)}\{t^{j}g_k\}$.
\item Reduce $U^{(\alpha)}\left(L_{\alpha}^{(B)}\right)\pmod{\ell^{B}}$ to get $L_{\alpha+1}^{(B)} = \displaystyle\sum_{\substack{j\in\mathbb{Z},\\ 0\le k\le v}} c_{\alpha+1,j,k}t^{j}g_k$.
\item Let $v_{\alpha+1}$ be the maximal power of $\ell$ that divides each $c_{\alpha +1,j,k}$.
\item Set $V = V\cup \{v_{\alpha+1}\}$.
\item Set $\alpha = \alpha+1$, and repeat.
\item Continue until $\alpha=B$.
\item The $v_{\alpha}$ will give the largest possible power of $\ell$ that divides $L_{\alpha}$.  We may either formulate a possible pattern, or check one already conjectured, for $0\le \alpha\le B$.
\end{enumerate}

\section{``Not an Actual Demonstration"}

Given a class of integer partitions with an effective eta quotient as a generating function, it is natural to search for various families of congruences, especially in the style of Ramanujan's famous results.  However, it is often difficult to gather compelling evidence for many prospective infinite family of congruences for computational reasons, so that better techniques are necessary.

The method developed in this paper gives us one such collection of techniques.  Indeed, it is very tempting to believe that, given a conjectured family of congruences, a complete proof of the family should be possible using the same techniques.  While we cannot completely rule out this possibility, we certainly may point out difficulties that are very likely insurmountable to this end.

The traditional means of actually proving the existence of a given infinite family of partition congruences, with respect to powers of a prime $\ell$, is the notion of $\ell$-adic convergence for a family of generating functions for each given case.  For example, in the case of Rogers--Ramanujan subpartitions, we have the functions

\begin{align*}
L_0 &= 1,\\
L_{\alpha} &= \Phi_{\alpha}\cdot \sum_{24n\equiv 1\bmod{5^{\alpha}}} a(n)q^{\left\lfloor n/5^{\alpha} \right\rfloor},
\end{align*} with $\Phi_{\alpha}$ suitably chosen.  To prove that $a(n)\equiv 0\pmod{5^{\alpha}}$ whenever $24n\equiv 1\pmod{5^{2\alpha}}$, we need to argue that for every $M\in\mathbb{Z}_{>0}$ there exists an $N\in\mathbb{Z}_{>0}$ such that for all $n\ge N$,

\begin{align*}
L_n\equiv 0\pmod{5^M}.
\end{align*}  In particular, we need to show that $N=\left\lfloor M/2 \right\rfloor$ will suffice.  This is done by very carefully constructing subspaces

\begin{align*}
S_0 &:= \left< 1,p_0 \right>_{\mathbb{Z}[t]},\\
S_1 &:= \left< 1,p_1 \right>_{\mathbb{Z}[t]},\\
S_{\alpha} &:= S_{\alpha\bmod 2}
\end{align*} of modular functions over $\Gamma_0(20)$, so that $L_{\alpha}\in S_{\alpha}$ for all $\alpha \ge 1$.  Moreover, we need to construct the functions $p_0, p_1, t$ so that $U^{(\alpha)}\left(p_{\alpha}\right)\in S_{\alpha+1}$, and that application of $U^{(1)}$ causes each element in $S_{2\alpha-1}$ to become divisible by a higher power of 5.

For the complete proof of this case, see \cite{Smoot}.

For this to work, the functions $p_{\alpha}, t$ must be very carefully selected, so that successive application of $U^{(\alpha)}$ will generate functions divisible by increasing powers of 5.  While the algebra basis that we have used in this paper is very powerful, it does not necessarily select functions $t$ with the property that the sequence 

\begin{align*}
\left(t,\ U^{(1)}\left(t\right),\ U^{(0)}\left(U^{(1)}\left(t\right)\right), ... \right)
\end{align*} converges 5-adically to 0.

We use the case of Section 3.2 as an example:

\begin{align*}
T &= \frac{\eta(\tau)^2\eta(4\tau)^2\eta(10\tau)^{8}}{\eta(5\tau)^2\eta(20\tau)^{10}} = \frac{1}{q^5}\frac{(q;q)^2_{\infty}(q^4;q^4)^2_{\infty}(q^{10};q^{10})^{8}_{\infty}}{(q^5;q^5)^2_{\infty}(q^{20};q^{20})^{10}_{\infty}}.
\end{align*}

Suppose we define $T_1 = U^{(1)}\left(T\right), T_{\alpha} = U^{(\alpha)}\left(T_{\alpha - 1}\right)$, for $\alpha \ge 1$.

\begin{align*}
T_1 &\equiv 4\frac{1}{T} + 2 G_1\frac{1}{T} + G_3\frac{1}{T} + G_2\frac{1}{T} \pmod{5},\\
T_2 &\equiv 3 G_1 \frac{1}{T} + 2 G_3 \frac{1}{T} \pmod{5},\\
T_3 &\equiv 3 G_1 \frac{1}{T} + 2 G_3 \frac{1}{T} \pmod{5},\\
T_4 &\equiv 3 \frac{1}{T} + 4 G_1 \frac{1}{T} + 2 G_2 \frac{1}{T} \pmod{5},\\
T_5 &\equiv 4 G_1 \frac{1}{T} + G_3 \frac{1}{T} \pmod{5},\\
T_6 &\equiv 4 \frac{1}{T} + 2 G_1 \frac{1}{T} + G_2 \frac{1}{T} \pmod{5},...
\end{align*} Continuing this through to $T_{14}$, we eventually have

\begin{align*}
T_{11} &\equiv 3 G_1 \frac{1}{T} + 2 G_3 \frac{1}{T} \pmod{5},\\
T_{12} &\equiv 3 \frac{1}{T} + 4 G_1 \frac{1}{T} + 2 G_2 \frac{1}{T} \pmod{5},\\
T_{13} &\equiv 4 G_1 \frac{1}{T} + G_3 \frac{1}{T} \pmod{5},\\
T_{14} &\equiv 4 \frac{1}{T} + 2 G_1 \frac{1}{T} + G_2 \frac{1}{T} \pmod{5},...
\end{align*}  Notice the repetition: $T_{11} \equiv T_{3}\pmod{5}$, $T_{12} \equiv T_4\pmod{5}$, $T_{13} \equiv T_5\pmod{5}$, and so on.  This sequence settles into a repeated pattern modulo 5, so it can never become $0\pmod{5}$, no matter how often we apply $U^{(\alpha)}$.  In other words, the sequence $(T_{\alpha})_{\alpha\ge 1}$ will not converge to 0 in the 5-adic sense.

A good analogy can be found with the question of convergence in the standard topology.  Suppose we have a sequence of functions 

\begin{align*}
(L_{\alpha})_{\alpha\ge 0},
\end{align*} and we suspect that 

\begin{align*}
\lim_{\alpha\rightarrow\infty} L_{\alpha} = 0.
\end{align*}  One way of proving this convergence is to find other sequences of functions, e.g. $(F_{\alpha})_{\alpha\ge 0}, (G_{\alpha})_{\alpha\ge 0}$ such that

\begin{align*}
L_{\alpha} = F_{\alpha} + G_{\alpha}.
\end{align*}  Now, to prove that $\lim_{\alpha\rightarrow\infty} L_{\alpha} = 0$, it is certainly sufficient to prove that

\begin{align*}
\lim_{\alpha\rightarrow\infty} F_{\alpha} = \lim_{\alpha\rightarrow\infty} G_{\alpha} = 0.
\end{align*}  However, it is not necessary at all---for instance, it is possible that 

\begin{align*}
\lim_{\alpha\rightarrow\infty} F_{\alpha} = 1,\ \lim_{\alpha\rightarrow\infty} G_{\alpha} = -1.
\end{align*}  If we want to prove convergence of $L_{\alpha}$ term-wise, it is clear that we need to carefully select our summands.  A similar principle holds, if we replace the notion of convergence in the standard topology with that of 5-adic topology, with our sequences of functions $(L_{\alpha})_{\alpha\ge 0}\subseteq\mathcal{M}(N)$.

We therefore take insight from Archimedes, that a method which allows us to gather evidence for a conjecture---even if it cannot give a proof---is often as important as the proof itself \cite{Archimedes}.  We hope that researchers may make fruitful use of our method to more efficiently apply the justifiably celebrated proof techniques which were developed by Watson, Atkin, Paule, Radu, and others.

\section{Acknowledgements}

The first author was supported by grant SFB F50-06 of the Austrian Science Fund (FWF) and by the strategic program ``Innovatives OÖ 2010 plus" by the Upper Austrian Government in the frame of project W1214-N15-DK6 of the FWF.  The second author was also supported by project W1214-N15-DK6 of the FWF.  We wish to thank the Austrian government for its generous support.

While we are proud of our results, the writing of this paper was a difficult and at times tedious project.  We are enormously grateful to the anonymous referees who each gave extremely helpful advice about the improvement of the writing, as well as the correction of multiple errors.

We are grateful to Dennis Eichhorn and James Sellers for their previous work \cite{Eichhorn} on verifying partition congruences, as well as to Ralph Hemmecke for his work with the first author in developing the software package 4ti2 \cite{4ti2}.  Both of these sources were enormously helpful to our work.

\end{document}